\documentclass[11pt,bezier]{article}
\usepackage{amsmath,amssymb,amsfonts,euscript}
\usepackage[usenames]{color}
\newpage

\usepackage{graphicx}
\usepackage{dsfont}
\usepackage{stmaryrd}
\usepackage{textcomp}
\usepackage{bbold}
\usepackage{pifont}
\usepackage{bbm}
\newtheorem{theo}{Theorem}
\newtheorem{defi}[theo]{Definition}
\newtheorem{lemma}[theo]{Lemma}
\newtheorem{pro}[theo]{Proposition}
\newtheorem{remark}[theo]{Remark}

\newenvironment{proof}{\trivlist\item[]\textbf{Proof.\ }}
                      {\endtrivlist}

\title{\bf\Large The Scale Invariant Wigner Spectrum Estimation of Gaussian Locally Self-Similar Processes}
\vspace{-1cm}
\author{Y. Maleki, S. Rezakhah \footnote{Faculty of Mathematics and Computer Science, Amirkabir University of Technology, 424 Hafez Avenue,
Tehran 15914, Iran. Email: ymaleki@aut.ac.ir, rezakhah@aut.ac.ir.}}

\date{June 2012}
\begin{document}
\maketitle

\begin{abstract}
We study locally self-similar processes (LSSPs) in Silverman's sense. By deriving the minimum mean-square optimal kernel within Cohen's class counterpart of time-frequency representations, we obtain an optimal estimation for the scale invariant Wigner spectrum (SIWS) of Gaussian LSSPs. The class of estimators is completely characterized in terms of kernels, so the optimal kernel minimizes the  mean-square error of the estimation. We obtain the SIWS estimation for two cases: global and local, where in the local case, the kernel is allowed to vary with time and frequency. We also introduce two generalizations of LSSPs: the locally self-similar chrip process and the multicomponent locally self-similar process, and obtain their optimal kernels. Finally, the performance and accuracy of the estimation is studied via simulation.\\ \quad \\
{\it Keywords}: Locally self-similar circularly symmetric Gaussian processes, scale invariant Wigner spectrum (SIWS), optimal estimation, time-frequency analysis.\\ \quad \\
{\it Mathematics Subjects Classification:} 62F10, 62F15, 62N02

\end{abstract}

\section{Introduction}
\par
Scale invariance (or self-similarity), once acknowledged as an important feature, has often been used as a fundamental property to interpret many natural and man-made phenomena \cite{fla}. Self-similar processes have been used successfully to model data exhibiting long memory and arising in a wide variety of fields, ranging from physics (turbulence, hydrology, solid-state physics, $\cdots$) or biology (DNA sequences, heat rate variability, auditory nerves spike trains, $\cdots$), to human-operated systems (telecommunications network traffic, image processing, pattern recognition,  finance, $\cdots$) \cite{fla7}. Because they may correspond to non-standard situations in signal processing or time-series analysis (non-stationarity, long range dependence, $\cdots$), scale invariant processes raise challenging problems in terms of analysis, synthesis, and processing (filtering, prediction, $\cdots$), and a number of specific tools have, however, developed over the years \cite{fla7}.
Even though there is no single definition of scale invariance, it is often described as a symmetry of the system relatively to a transformation of a scale, that is mainly a dilation or a contraction (up to some renormalization) of the system parameters \cite{fla}. In mathematical expression, a process is scale invariant (or self-similar, denoted 'H-ss') \cite{fla}, if for all $\lambda>0$,
\begin{equation}
\{(D_{H,\lambda}X)(t):=\lambda^{-H}X(\lambda t),t>0\}\equiv \{ X(t),t>0\}\label{hss}
\end{equation}
(where $\equiv$ means equal in all finite dimensional distributions). The index $H$ characterizes the self-similar behavior of the process, and a very large variety of methods has been proposed in the literature for estimating it \cite{bardet}, \cite{beran}, \cite{corjoly2001}.
\par
Sometimes it happens that the H-ss model is not quite adequate for real world phenomenon, and it would be useful to consider more general classes of stochastic processes, which are characterized by more than one parameter, but still preserve some of the good properties of scale invariant processes. Therefore, larger classes of stochastic processes have been introduced, for example the class of multiplicative harmonizable processes introduced by Borgnat and Flandrin \cite{fla1}, which admits the use of techniques for harmonic analysis. Locally self-similar processes (LSSPs) constitute another class of extensions of H-ss processes. This class of processes can be used to describe physical systems for which statistical characteristics change slowly in time. Based on different situations, several definitions have been raised for LSSPs. Out of the large literature, we mention the work of Flandrin et {\it{al}} \cite{fla5}, \cite{fla3} in introduction of LSSPs as Lamperti transformation of locally stationary processes in Silverman's sense \cite{silverman}. Besides, Flandrin and Goncalves \cite{fla6} and Goncalves and Flandrin  \cite{gon} proposed LSSPs in a sense that the self-similarity parameter, $H(t)$, is allowed to vary with time, and discuss their applications.  Cavanaugh  et {\it{al}} \cite{cavanaugh1},  Courjouly \cite{corjoly2001}, Goncalves and Abry \cite{abry},  Kent and Wood \cite{kent}, Stoev et {\it{al}} \cite{stoev} and Wang et {\it{al}} \cite{cavanaugh2} studied LSSPs in the latter sense, and estimated the local Hurst parameter $H(t)$. In addition, Wang et {\it{al}} \cite{cavanaugh2} and Constantine \cite{cons} investigated fitting LSSPs to a geophysical and  aerothermal turbulence  time series respectively. Boufoussi et {\it{al}} \cite{bou} studied path properties of a class of local asymptotic self-similar process. Moreover, Muniandy and Lim \cite{mou} considered modeling of LSSPs using multifractional Brownian motion.
We also mention the work of Istas and Lacauxy \cite{istas} on locally self-similar fractional random fields.
 The road to LSSPs analysis was followed especially in time-domain and Hurst parameter estimation.
\par
In this paper, we consider LSSPs in Silverman's sense \cite{fla}, \cite{silverman}, and we study their spectral analysis.
In this sense, a stochastic process is LSSP \cite{fla}, if its covariance function, $R_X(t,s)=E(X(t)X^*(s))$, has the simple form
$$R_X(t,s)=(ts)^Hq(\ln \sqrt{ts})C_X(t/s):=Q(\sqrt{ts})C_X(t/s),$$
where $Q(t):=t^{2H}q(\ln{t})$, i.e., the covariance function can be decomposed into a covariance function of an H-ss process multiplied by a modulatory function.

Although there is no "universal definition" for nonstationary spectrum, some definitions have been proposed for restricted classes. More notable ones are the evolutionary spectrum (ES) proposed by Priestley \cite{prist} for the class of oscillatory processes, and the Wigner-Ville spectrum (WVS) proposed by Martin \cite{martin} for the class of harmonizable processes \cite{say}. Because of a few important mathematical properties of the WVS over the ES and other nonstationary spectrums, such as uniqueness and that the WVS is explicitly defined in terms of the covariance function; the WVS has been more desirable than other spectrums \cite{say}. So, we adopt the WVS as our definition for the nonstationary spectrum and address the problem of estimating it for Gaussian LSSPs. However, the scale invariant property of LSSPs, makes such processes different from other nonstationary ones. Therefore, a particular type of the Wigner spectrum should be considered, which is compatible with this property. In this case, for scale invariant signals, a bilinear time-dependent extension of the Mellin transform  proposed by Marinovic \cite{marinovic} and Altes \cite{altes}, which will be referred to as the scale invariant Wigner Distribution (SIWD). Flandrin in \cite{fla5} provided some results of the SIWD and extended it to stochastic processes.

The scale invariant Wigner spectrum (SIWS) \cite{fla5} of a process $X(t)$ is defined as an expectation of the SIWD, and it is a function of time $t$ and frequency $\xi$, which reads as
\begin{equation}
W_{E,X}(t,\xi)=E\{\int_{-\infty}^{\infty}X(t\sqrt{\tau})X^*(t/\sqrt{\tau}) \tau^{-i2\pi \xi-1} d\tau \}
\end{equation}
$$\;\;\quad \quad=\int_{-\infty}^{\infty}R_X(t\sqrt{\tau}, t/\sqrt{\tau}) \tau^{-i2\pi \xi-1} d\tau. $$
The usefulness of the SIWS has been emphasized for revealing scale invariant feature in some nonstationary processes.
By contraction, the SIWS describes the time evolution of Mellin's variable in a similar way as the WVS does for the (Fourier) frequency variable \cite{fla}. To avoid confusion, the classical Wigner-Ville spectrum will be referred to as WVS, and the scale invariant spectrum will be referred to as SIWS.

In this paper, we estimate the SIWS of a Gaussian LSSP, using the mean-square optimal kernel method. The optimal kernel method is proposed by  Sayeed and Jones \cite{say} for nonstationary spectrum estimation; and is based on finding the best estimator that optimizes the bias-variance trade-off in the sense of giving minimum mean-square error (MMSE). However, as we mentioned, the scale invariant feature of the process and consequently, the spectrum, makes the SIWS estimation method different from the other nonstationary spectrum estimations. So, we modify the method to be reconcile to our estimation problem. Moreover, as the class of estimators, we choose Cohen's class counterpart  of time-frequency representations (TFR's), which is proposed by Flandrin \cite{fla5} as a generalization of Cohen's class \cite{cohen} for scale invariant signals. Since the Cohen's class counterpart is completely characterized in terms of kernels, the problem of estimation is reduced to finding the "best" kernel in the sense of giving the minimum mean-square error. The SIWS estimator is obtained in two cases: global case, and local case; where in the local case, the optimal kernel is allowed to vary with time and frequency.

Now, we present an outline of the paper. Next section, concerns the background on self-similar, multiplicative harmonizable, circularly symmetric processes. Section 3 treats locally self-similar processes, and we study necessary conditions on two functions, that constitute the covariance function of a LSSP. In Section 4, we discuss the global and local estimation problems of SIWS for Gaussian LSSPs. In Section 5, we introduce two extensions of LSSPs; also, we present one example and the performance of the method is studied via simulations.

\section{Preliminaries}
The Fourier transform of a function $f \in L^1
(\mathbbm{R})$ is defined and denoted by \cite{wal2}
$$(\mathcal{F}f)(\xi):=\hat{f}(\xi):=\int_{\mathbbm{R}}f(t) e^{-i2\pi t\xi}dt.$$

A function $f: \mathbbm{R}^2  \rightarrow \mathbbm{C}$ is nonnegative definite  \cite{berg}, denoted $f \in NND(\mathbbm{R}^2)$, if
\begin{equation}
\sum_{j,k=1}^n f(t_j,t_k)z_j z^*_k \geq 0 \quad \quad \forall \{t_j\}_{j=1}^n \subset \mathbbm{R}, \quad \{z_j\}_{j=1}^n \subset \mathbbm{C}, \quad n>0,\label{1.8}
\end{equation}
and it is called weakly nonnegative definite \cite{berg}, \cite{wal4}, denoted $f \in WNND(\mathbbm{R}^2)$, if $f$ is bounded, measurable, and
\begin{equation}
\int\int_{\mathbbm{R}^2} f(t,s)\vartheta(t) \vartheta^*(s) dt ds \geq 0 \quad \quad \forall\vartheta \in C_c(\mathbbm{R}).\label{1.9}
\end{equation}
where $C_c(\mathbbm{R})$ denotes the compactly supported continuous functions.
If $f$ is continuous and bounded, then $f \in NND(\mathbbm{R}^2)$ if and only if $f \in WNND(\mathbbm{R}^2) $ \cite{berg}. For generalized functions, we mean by $f \in NND(\mathcal{S}(\mathbbm{R}^2))$ that $f \in \mathcal{S'}(\mathbbm{R}^2)$ and
\begin{equation}
(f,\vartheta \varotimes \vartheta^*)\geq 0 \quad \forall \vartheta \in \mathcal{S}(\mathbbm{R}),\label{pre}
\end{equation}
where $\mathcal{S}(\mathbbm{R}^d)$ is the Schwartz space of smooth functions such that a derivative of any order multiplied with any polynomial is uniformly bounded, and $\mathcal{S'}(\mathbbm{R}^d)$ is its dual, the tempered distributions \cite{hormander}, \cite{wal4}; and the bracket $(f,g)=\int_{R^d} f(x)g^*(x)dx$ denotes the inner product on $L^2(\mathbbm{R}^d)$.
 A function $f$ belongs to $NND(\mathbbm{R}^2)$ if and only if it is the covariance function of a mean-square continuous stochastic process \cite{loeve}.

A function $f: \mathbbm{R} \rightarrow \mathbbm{C}$ of one real variable is nonnegative definite, denoted $f \in NND(\mathbbm{R})$, if
$$\sum_{j,k=1}^n f(t_j-t_k)z_j z^*_k \geq 0 \quad \quad \forall \{t_j\}_{j=1}^n \subset \mathbbm{R}, \quad \{z_j\}_{j=1}^n \subset \mathbbm{C}, \quad n>0.$$

We will use some results about measure theory \cite{rudin}. The domain of a complex-valued measure $\mu$ on $\mathbbm{R}$ is a class of subsets of $\mathbbm{R}$, which is closed under countable unions and complement, and it includes the empty set. We use the Borel $\sigma$-algebra, denoted $\mathcal{B}(\mathbbm{R})$, as a domain for measures. A measure fulfills $\mu(\emptyset)=0$. It is countably additive in the sense of $\mu(\bigcup_{k=1}^\infty{A_k})=\sum_{k=1}^\infty \mu(A_k)$ when $\{A_k\}_{k=1}^\infty$ are pair wise disjoint. The total variation of a measure $\mu$ is defined as:
$$|\mu|(A)= sup \sum_{k=1}^\infty |\mu(A_k)|, \qquad A \in \mathcal{B}(\mathbbm{R})  $$
where the supremum is taken over all $\{A_k\}_{k=1}^\infty \subset \mathcal{B}$ such that $\bigcup_{k=1}^\infty A_k=A$ and  $\{A_k\}_{k=1}^\infty$ are pairwise disjoint. The total variation, $|\mu|$, is a finite measure if  $\mu$ is a finite measure.

\subsection{Self-Similar processes}
A stochastic process $\{X(t), t>0\}$ is said to be (statistically) self-similar of index $H$ \cite{fla} (or scale-invariant, denoted 'H-ss'), if
for any $\lambda>0$, Eq (\ref{hss}) is satisfied; i.e., a self-similar process is invariant under any renormalized dilation operator $D_{H,\lambda}$ by a scale factor $\lambda$ \cite{fla}, \cite{fla3}.

A random process $\{X(t), t>0\}$ is called wide sense self-similar with parameter $H$ \cite{kashyap}, if it satisfies the following conditions:

\begin{enumerate}
\item $E(X(t))=\lambda^{-H}E(X(\lambda t))$ for all $t, \lambda>0$
\item $E(X^2(t))<\infty$ for each $t>0$
\item $E(X(t_1)X^*(t_2))=\lambda^{-2H}E(X(\lambda t_1)X^*(\lambda t_2))$ for all $t_1, t_2, \lambda>0.$

\end{enumerate}

\begin{pro}
Any wide sense H-ss processes $\{X(t), t>0\}$ has a covariance function of the form \cite{fla}
\begin{equation}
R_X(t,s)=(ts)^H C_X(\frac{t}{s}),\label{1'}
\end{equation}
for any $t,s>0$, and $C_X(.)\in NND(\mathbbm{R}^+)$.
\end{pro}

\begin{proof}
See Flandrin \cite{fla3}.
\end{proof}

In the case of self-similar processes, the Mellin transform plays, with respect to scaling, a role similar to that played by the Fourier transform with respect to shifting \cite{fla3}. The Mellin transform \cite{fla}, \cite{desana} of a function $g(t)$ is defined by
\begin{equation}
{\check{\mathcal{G}}}(s):=(\mathcal{M}g)(s)=\int_0^\infty g(t) t^{-s-1}dt,\label{mellin}
\end{equation}
where $s=\beta+i \omega$, $\beta, \omega \in \mathbbm{R}$.

For a function $g$ of several variables we denote partial Mellin transform with respect to variables indexed by $j , k$, by $\mathcal{M}_{j,k}g$.

\subsection{Harmonizable, Multiplicative Harmonizable, and Circularly Symmetric Gaussian processes}
We present here some definitions about multiplicative harmonizable processes from Flandrin et {\it{al}} \cite{fla}, \cite{fla1}, \cite{fla3}. A process $\{X(t), t>0\}$ is called multiplicative harmonizable if its covariance function has the representation
\begin{equation}
R_X(t,s)=\int_{-\infty}^\infty \int_{-\infty}^\infty t^{H+i2\pi\beta} s^{H-i2\pi\sigma} m(d\beta,d\sigma), \label{2'}
\end{equation}
where $m$ is a measure on $\mathbbm{R}^2$ of bounded total variation, and is called the Mellin spectral distribution function. Furthermore, this spectral distribution function satisfies
$$m(\beta,\sigma)=\int_0^\infty \int_0^\infty t^{-H-i2\pi\beta-1} s^{-H+i2\pi\sigma-1} R_X(t,s) dt ds.$$
A necessary and sufficient condition for (\ref{2'}) to hold is that
$\int_{-\infty}^\infty \int_{-\infty}^\infty|m(d\beta,d\sigma)|<\infty$, which is adopted from Loeve's condition for harmonizability \cite{loeve}. A multiplicative harmonizable process $X$ admits a spectral representation on a Mellin basis as:
\begin{equation}
X(t)=\int_{-\infty}^\infty t^{H+i2\pi\beta} d\widetilde{X}(\beta),\label{multi}
\end{equation}
where $\widetilde{X}$ is an $L^2(\mathbbm{P})$-valued measure, called the spectral measure or spectral process. The connection between $\widetilde{X}$ and $m$ is:
$$E(d\widetilde{X}(\beta)d{\widetilde{X}^*(\sigma)})=m(d\beta,d\sigma).$$
For multiplicative harmonizable H-ss  processes, $m$ has support on the diagonal, and it is a nonnegative bounded measure. Therefore, there exists a non-negative bounded measure $\widetilde{C}$ such that
\begin{equation}
C_X(t/s)=\int_{-\infty}^\infty (t/s)^{i2\pi\beta}\widetilde{C}(d\beta),\label{multihss}
\end{equation}
since, by (\ref{multi}) and $m(d\beta, d\sigma)=\bigg{\{}^{\widetilde{C}(d\beta) \;\;\; \beta=\sigma}_{0\;\;\;\;\;\;\;\;\;\;\beta\neq\sigma }$
$$E(X(t)X^*(s))=\int_{-\infty}^\infty \int_{-\infty}^\infty t^{H+i2\pi\beta} s^{H-i2\pi\sigma} E(d\widetilde{X}(\beta)d\widetilde{X}^*(\sigma))$$
$$=(ts)^H \int_{-\infty}^\infty (t/s)^{i2\pi\beta}\widetilde{C}(d\beta),\quad $$
\noindent by comparison to (\ref{1'}), Eq (\ref{multihss}) is achieved.

\begin{defi}
For a circularly symmetric or proper process \cite{pic}, \cite{sch} $X(t)$, the processes
$$\{e^{i\theta}X(t)\}_{\theta \in [0,2\pi)}$$
\noindent are identically distributed for all ${\theta \in [0,2\pi)}$.

According to Grettenberg's theorem \cite{rao}, for a circularly symmetric process, $E(X(t))\equiv 0$ and $E(X(t)X(s))=0 \quad  \forall(t,s)\in \mathbbm{R}^2$ \cite{miller}.
\end{defi}

Let $X_1, \cdots , X_n$ be complex-valued zero-mean jointly Gaussian stochastic variables. Then, according to Wick's theorem \cite{wick},
\begin{equation}
E(X_1 \cdots  X_n)=\sum \prod E(X_{i_k}X_{j_k}),\label{2.8}
\end{equation}
where the sum is over all partitions of $\{1, \cdots , n\}$ into disjoint pairs $\{i_k,j_k\}$. Thus for $n=4$ we have
\begin{equation}
E(X_1X_2X_3X_4)=E(X_1X_2)E(X_3X_4)+E(X_1X_3)E(X_2X_4)+E(X_1X_4)E(X_2X_3).\label{2.9}
\end{equation}
The formula (\ref{2.9}) sometimes called $\mathit{Isserlis' theorem}$ is valid for any zero mean complex-valued Gaussian stochastic variables $X_1, X_2, X_3, X_4$ \cite{hel}.

\noindent The following fourth-order moment function will be needed later on.
\begin{remark}
Let
\begin{equation}
k(t_1, \tau_1, t_2, \tau_2):=E\{X(t_1\sqrt{\tau_1})X^*(t_1/\sqrt{\tau_1})X^*(t_2\sqrt{\tau_2})X(t_2/\sqrt{\tau_2})\}.
\end{equation}
\noindent Then, from (\ref{2.9}), for real-valued Gaussian processes, we have that
$$k(t_1, \tau_1, t_2, \tau_2)=R_X(t_1\sqrt{\tau_1},t_1/\sqrt{\tau_1})R_X(t_2\sqrt{\tau_2} ,t_2/\sqrt{\tau_2})$$
$$+R_X(t_1\sqrt{\tau_1},t_2\sqrt{\tau_2})R_X(t_1/\sqrt{\tau_1},t_2/\sqrt{\tau_2})$$
\begin{equation}
+R_X(t_1\sqrt{\tau_1},t_2/\sqrt{\tau_2})R_X(t_1/\sqrt{\tau_1},t_2\sqrt{\tau_2}),\label{8}
\end{equation}
and for circularly symmetric Gaussian processes,
$$k(t_1, \tau_1, t_2, \tau_2)=R_X(t_1\sqrt{\tau_1},t_1/\sqrt{\tau_1})R^*_X(t_2\sqrt{\tau_2} ,t_2/\sqrt{\tau_2})$$
\begin{equation}
+R_X(t_1\sqrt{\tau_1},t_2\sqrt{\tau_2})R^*_X(t_1/\sqrt{\tau_1},t_2/\sqrt{\tau_2})\label{9}.
\end{equation}
\end{remark}

\section{Locally Self-Similar Processes}
Locally self-similar processes (LSSPs) in Silverman's sense \cite{silverman}, were introduced by Flandrin \cite{fla3} as an extension of H-ss processes. In this Section, we study properties of LSSPs in time and time-frequency domain.

\begin{defi}
A locally self-similar process (LSSP) \cite{fla}, \cite{fla3}, is a complex-valued stochastic process $\{X(t), t>0\}$ whose covariance function $R_X(t,s)=E(X(t)X^*(s))$, has the form
\begin{equation}
R_X(t,s)=(ts)^H q(\ln \sqrt{ts}) C_X(\frac{t}{s})\label{1''}
\end{equation}
where $q$ and $C_X$ are complex-valued functions, $q$ must have a constant sign which we assume positive, and $C_X \in  NND(\mathbbm{R}^+)$. If we define $Q(t):=t^{2H}q(\ln t)$, and the coordinate transformation on $\mathbbm{R}^{+2}$ as
\begin{equation}
k(t,s):=(t\sqrt{s},t/\sqrt{s})\Longleftrightarrow k^{-1}(t,s)=(\sqrt{ts}, \frac{t}{s}),\label{k}
\end{equation}
then (\ref{1''}) can be written as
\begin{equation}
R_X(t,s)=Q(\sqrt{ts})C_X(\frac{t}{s}), \label{rxk}
\end{equation}
we may write this relation as
$$R_X(t,s)=Q \varotimes C_X \circ k^{-1}(t,s).$$
If $R_X \in NND(\mathbbm{R}^{+2})$ and $R_X=Q \varotimes C_X \circ k^{-1}$, we write $(Q,C_X)\in LSS(\mathbbm{R}^+)$.
\end{defi}

\noindent We will always understand that the stochastic process is nonzero, i.e., $Q, C_X \neq 0$. Let $(Q,C_X)\in LSS(\mathbbm{R}^+)$, and let $X(t)$ be a locally self-similar process with covariance function $Q \varotimes C_X \circ k^{-1}$. Then $$Q(t)C_X(1)=E|X(t)|^2$$
\noindent describes the variation of the process over time, and $$C_X(\tau)/C_X(1)=E(X(t\sqrt{\tau})X^*(t/\sqrt{\tau}))\big{/}E|X(t)|^2,$$
\noindent describes the "local covariance" (at time $t$, for any $t \in \mathbbm{R}^+$ such that $E|X(t)|^2\neq 0$ ), of the stochastic process $X(t)$. Since $X$ is assumed to be nonzero, we have $C_X(1) \neq 0$, and we may without loss assume that
\begin{equation}
C_X(1)=1 \quad for \quad (Q,C_X)\in LSS(\mathbbm{R}^+). \label{2.2}
\end{equation}
Thus, $Q(t)\geq 0$ for all $t \in \mathbbm{R}^+.$\\

We also assume that $Q, C_X$ are continuous. Continuity of $Q$ and $C_X$ is equivalent to mean-square continuity of the process $X(t)$. In fact, everywhere continuity of $R_X$ is implied by continuity on the diagonal \cite{loeve}. Thus, if $Q$ is continuous everywhere and $C_X$ is continuous in $1$, then $R_X$ is continuous everywhere, and the process is mean-square continuous.\\

The Cauchy-Schwarz inequality for the Hilbert space $L^2(\Omega)$ gives $$|R_X(t,s)|^2\leq R_X(t,t)R_X(s,s),$$
 and is a necessary condition for a function $R_X$ to be the covariance function of a stochastic process \cite{loeve}. If
$(Q,C_X)\in LSS(\mathbbm{R}^+)$, we have due to the normalization (\ref{2.2}) which implies $Q \geq 0$,

$$\qquad\qquad\qquad Q^2(t)|C_X(\tau)|^2\leq Q(t\sqrt{\tau})Q(t/\sqrt{\tau}), \qquad t, \tau \in \mathbbm{R}^+.$$

\noindent By restriction to $t=1$, we obtain the inequality
$Q^2(1)|C_X(\tau^2)|^2\leq Q(\tau)Q(1/\tau)$, or equivalently
\begin{equation}
Q^2(1)|C_X(\tau^2)|^2\leq Q(\tau)Q(1/\tau),\qquad
\end{equation}
which says that $C_X$ is necessarily bounded in terms of $Q(\tau)Q(1/\tau)$.\\

We are interested in conditions on $Q$ and $C_X$ that are necessary and sufficient for $(Q,C_X)\in LSS(\mathcal{S}(\mathbbm{R}^+))$, which is equivalent to $(Q, C_X)\in LSS(\mathbbm{R}^+)$ when $Q$ and $C_X$ are continuous and polynomially bounded. The following simple lemma expresses the definition $(Q,C_X)\in LSS(\mathbbm{R}^+)$ in terms of time-frequency analysis.
\begin{lemma}
If $Q,C_X \in \mathcal{S}(\mathbbm{R}^+)$ then
\begin{equation}
(Q\varotimes C_X) \in LSS(\mathcal{S}(\mathbbm{R}^+)) \Longleftrightarrow \int_0^\infty \int_{-\infty}^\infty t Q(t) \breve{C}_X(\xi) W_f(t/t_0,\xi-\xi_0) dt d\xi \geq 0, \quad f \in \mathcal{S}(\mathbbm{R}^+)\label{lem}
\end{equation}
where $W_f(t,\xi)=\int_0^\infty f(t\sqrt{\tau})f^*(t/\sqrt{\tau})\tau^{-i2\pi\xi-1}d\tau$, is the scale invariant Wigner distribution. 
\end{lemma}

\begin{proof}
By $(Q,C_X)\in LSS(\mathcal{S}(\mathbbm{R}^+))$, we mean $R_X=Q\varotimes C_x \circ k^{-1} \in NND(\mathbbm{R}^{+2})$ with continuous and polynomially bounded functions $Q$ and $C_X$. Then, by (\ref{pre}), $(Q\varotimes C_x \circ k^{-1},\vartheta \varotimes \vartheta^*)\geq 0,\; \forall \vartheta \in \mathcal{S}(\mathbbm{R}^+)$. Thus,

$$(Q\varotimes C_X \circ k^{-1},\vartheta \varotimes \vartheta^*)=\int_0^\infty\int_0^\infty Q\varotimes C_X \circ k^{-1}(u,v)\;\vartheta^* \varotimes \vartheta(u,v)du dv\qquad$$

$$\qquad\qquad\qquad=\int_0^\infty\int_0^\infty Q(\sqrt{uv})C_X(u/v)\vartheta^*(u) \vartheta(v) du dv $$

$$\qquad\qquad\qquad\qquad\;\;\;\;\;\qquad\qquad\;\;=\int_0^\infty \int_{-\infty}^\infty t Q(t) \breve{C}_X(\xi)\{\int_0^\infty \vartheta^*(t\sqrt{\tau}) \vartheta(t/\sqrt{\tau}) \tau^{i\xi-1} d\tau\} d\xi dt\geq 0$$
where $t:=\sqrt{uv}, \tau:=u/v$ and $\breve{C}_X:=\mathcal{M}^{-1}C_X$. By  real-valuedness of $W_\vartheta$, we have that
$$(Q\varotimes C_x \circ k^{-1},\vartheta \varotimes \vartheta^* )=\int_0^\infty \int_{-\infty}^\infty t Q(t) \breve{C}_X(\xi)W_\vartheta(t,\xi)d\xi dt\geq 0.$$

\noindent Let $f \in \mathcal{S}(\mathbbm{R}^+)$ and set $\vartheta(t):=t^{i2\pi\xi_0}f(\frac{t}{t_0}) \in \mathcal{S}(\mathbbm{R}^+)$, then $W_\vartheta(t, \xi)=W_f(t/t_0, \xi-\xi_0)$, and (\ref{lem}) is satisfied.

$\hspace{6in} \boxempty$
\end{proof}

\subsection{Multiplicative harmonizable locally self-similar processes}
In the remainder of this section, we study locally self-similar processes that are multiplicative harmonizable. Let there exists a bounded measure $\widetilde{C}$ such that (\ref{multihss}) holds, and also there exists a bounded measure $\widetilde{Q}$ such that $Q(t)=\int_{-\infty}^\infty t^{2H+i2\pi\xi}\widetilde{Q}(d\xi)$. Then for a LSSP, (\ref{2'}) takes the form
$$R_X(t,s)=Q(\sqrt{ts}) C_X(t/s)=\int_{-\infty}^\infty \int_{-\infty}^\infty  (\sqrt{ts})^{2H+i2\pi\xi} (t/s)^{i2\pi\omega}\widetilde{Q}(d\xi) \widetilde{C}(d\omega)\qquad\qquad$$
$$\qquad\qquad\qquad\qquad\qquad=\int_{-\infty}^\infty \int_{-\infty}^\infty t^{H+i2\pi(\omega+\xi/2)}s^{H-i2\pi(\omega-\xi/2)}\widetilde{Q}(d\xi) \widetilde{C}(d\omega).$$
\noindent By the change of variables $\beta:=\omega+\xi/2$ and $\sigma:=\omega-\xi/2$, we have that
$$R_X(t,s)=\int_{-\infty}^\infty \int_{-\infty}^\infty t^{H+i2\pi\beta} s^{H-i2\pi\sigma} \widetilde{Q}(\beta-\sigma) \widetilde{C}(\frac{\beta+\sigma}{2})d\beta d\sigma.$$
By comparison to (\ref{2'}), it follows that the spectral distribution function, $m$, can be written as a product of bounded measures $\widetilde{C}$ and $\widetilde{Q}$ :
$$m(\beta,\sigma)= \widetilde{Q}(\beta-\sigma) \widetilde{C}(\frac{\beta+\sigma}{2}).$$

\section{The Scale Invariant Wigner Spectrum Estimation}
 In this section, we study the SIWS estimation of a zero mean real-valued, and complex-valued Gaussian LSSP defined on a positive half of a continuous time axis. As the class of estimators, we use the Cohen's Class Counterpart \cite{fla5} of time-frequency representations (TFRs) for scale invariant signals. The class of estimators, closely parallels the conventional Cohen's class \cite{cohen} and shares with it some of its most interesting properties, namely those concerning the usefulness and versatility of distributions associated to separable smoothing functions \cite{fla5}. We study estimators in this class that are optimal in the sense of giving minimum mean-square error (MMSE).

\subsection{The Scale Invariant Wigner Spectrum}
Given a scale invariant signal $x$ defined on $\mathbbm{R}^+$, the scale invariant Wigner distribution (SIWD)  \cite{altes},  \cite{marinovic}, is defined by
\begin{equation}
W_{x}(t,\xi)=\int_0^\infty x(t\sqrt{\tau})x^*(t/\sqrt{\tau})\tau^{-i2\pi\xi-1}d\tau\label{wx}
\end{equation}
$$=\mathcal{M}_2(x\varotimes x^*\circ k)(t,\tau),\;\;\;$$
which can be interpreted as a distribution of the signal's energy over the time-frequency domain $(t,\xi)\in \mathbbm{R}^+ \times \mathbbm{R}$. However, this interpretation is bothered by the fact that $W_x$ is rarely non-negative everywhere. As a remedy, one can smooth $W_x$ with a kernel according to $W_x \varodot \Phi$, which is called a time-frequency representation in  Cohen's class counterpart determined by $P_x$ \cite{fla5}, \cite{wal3}.

If $X$ is a stochastic process, then $W_x$ is called the scale invariant Wigner process, and is denoted by $W_X$. In this case, the integral (\ref{wx}) is a stochastic mean-square Riemann integral \cite{loeve}; and the  scale invariant Wigner spectrum (SIWS) is defined by the expectation of $W_X$:
\begin{equation}
W_{E,X}(t,\xi)=E\{\int_0^\infty X(t\sqrt{\tau})X^*(t/\sqrt{\tau})\tau^{-i2\pi\xi-1}d\tau\} \label{wvs}
\end{equation}
$$\qquad=\int_0^\infty R_X(t\sqrt{\tau},t/\sqrt{\tau})\tau^{-i2\pi\xi-1}d\tau$$
$$=(\mathcal{M}_{\tau}R_{X})(i2\pi\xi)\quad\qquad\quad\;\;\;$$
where $R_X$ is the covariance function of a process $\{X(t), t>0\}$ \cite{fla}, \cite{fla5}.

\begin{pro}
In the case of multiplicative harmonizable processes, the SIWS can be equivalently expressed in terms of  the spectral covariance function, $m(\beta,\sigma)$, as
\begin{equation}
W_{E,X}(t,\xi)=\int_{-\infty}^\infty t^{2H+i2\pi\theta} m(\xi+\theta/2,\xi-\theta/2) d\theta.\label{wteta}
\end{equation}
\end{pro}
\begin{proof}
See A1.
\end{proof}

 The scale invariant ambiguity function (SIAF) \cite{fla5}, can be obtained from the SIWD by Mellin duality:
\begin{equation}
A_x(\theta,\tau)=\int_0^\infty x(t\sqrt{\tau})x^*(t/\sqrt{\tau}) t^{-i2\pi\theta-1}dt\label{ax}
\end{equation}
$$=\mathcal{M}_1\mathcal{M}_2^{-1} W_x(t, \xi),\quad\quad$$

\noindent and for a stochastic process $X$, the Expected scale invariant ambiguity function (ESIAF) is defined by:
$$A_{E,X}(\theta,\tau):=E\{A_X(\theta,\tau)\}=\int_0^\infty R_X(t\sqrt{\tau},t/\sqrt{\tau})t^{-i2\pi\theta-1}dt$$
$$\qquad\qquad\;\;\;\;\;=(\mathcal{M}_tR_X)(i2\pi\theta)$$
\begin{equation}
\qquad \qquad\qquad\;\;\;\;\;=\mathcal{M}_1\mathcal{M}_2^{-1}W_{E,X}(t,\xi).\label{aex}
\end{equation}\\

\noindent For a locally self-similar process, by (\ref{rxk}),
$$R_X(t\sqrt{\tau},t/\sqrt{\tau})=Q(t)C_X(\tau),$$
so by (\ref{wvs}), one can easily verify that the SIWS of a LSSP can be represented by:
$$W_{E,X}(t,\xi)=\int_0^\infty Q(t) C_X(\tau) \tau^{-i2\pi\xi-1} d\tau$$
\begin{equation}
\quad\;\; =Q(t)(\mathcal{M}C_X)(i2\pi\xi),\label{welssp}
\end{equation}

\noindent and the ESIAF for this class is
$$A_{E,X}(\theta,\tau)=\int_0^\infty Q(t) C_X(\tau) t^{-i2\pi \theta -1} dt$$
\begin{equation}
\qquad =C_X(\tau)(\mathcal{M}Q)(i2\pi\theta).\label{aex1}
\end{equation}

\subsection{The Cohen's Class Counterpart}
The Cohen's class counterpart  of bilinear time-frequency representations (TFRs) of scale invariant signals is represented as
\begin{equation}
P_x(t, \xi)=\int_{-\infty}^\infty \int_0^\infty W_x(\frac{t}{s},\xi-\eta)\Phi(s, \eta) \frac{ds}{s} d\eta,\label{fla}
\end{equation}
where $\Phi$ is the 2-D kernel that completely characterized the particular TFR $P_x$, and $W_x$ is the SIWD \cite{fla5}. Eq (\ref{fla}) can equivalently be represented as \cite{fla5}
\begin{equation}
P_x(t,\xi)=\int_{-\infty}^\infty \int_0^\infty A_x(\theta, \tau)\phi(\theta, \tau)t^{i2\pi\theta}\tau^{-i2\pi\xi-1} d\tau d\theta \label{13'}
\end{equation}
$$=\mathcal{M}_1^{-1}\mathcal{M}_2\{A_x(\theta, \tau)\phi(\theta, \tau)\},\quad \;\;\;\;\;\;\;\;$$
where $\phi(\theta, \tau)=\int_{-\infty}^\infty \int_0^\infty \Phi(t,\xi) t^{-i2\pi\theta-1} \tau^{i2\pi\xi} dt d\xi.$
This class, closely parallels the  Cohen's class \cite{cohen}, and it shares with it some of its most interesting properties \cite{fla5}.

If $x$ denotes a realization of a scale invariant process $\{X(t),t>0\}$, then the integrals defined in (\ref{fla}) and (\ref{13'}) become stochastic integrals and will be interpreted as mean-square (m.s.) integrals \cite{loeve}, \cite{say}. We assume that the kernel $\phi$ is chosen such that the existence of the SIWD as a m.s. integral, implies the existence of $P_X$ as a m.s. integral.

\subsection{The SIWS Estimation}
Here, we study optimal kernel estimation of the SIWS, for  Gaussian LSSPs in two cases: (i) the local case in which the kernel $\phi$ is allowed to vary with time and frequency in order to better track the nonstationary structure of the process, and (ii) the global case, where the kernel $\phi$ is assumed to be independent of time and frequency.\\

\subsubsection{The Local Estimation}
Let $\{X(t),t>0\}$ be a zero mean real-valued, or complex-valued circularly symmetric, Gaussian process. Then, (\ref{wx}), (\ref{ax}) and (\ref{13'}) will be the stochastic Riemann integrals, denoted by $W_X$, $A_X$ and $P_X$.
 The stochastic integrals exist for all argument values \cite{loeve}, and are second order stochastic processes \cite{wal3}. $P_X(t,\xi)$, by definition, is a member of Cohen's class counterpart of TFRs.

Now, we consider the local SIWS estimation problem.  We use MMSE estimator where the optimality criterion of the estimation will be the minimization of the mean-square value of the estimation error. So, our objective is finding the optimal kernel $\phi_{opt}^{(t,\xi)}$, which minimizes the mean-square error
\begin{equation}
J(\phi^{(t,\xi)})=E\{\big {|}P_X(t,\xi)-W_{E,X}(t,\xi)\big{|}^2 \}\label{14'}
\end{equation}
\noindent for each $(t,\xi)$. This is a problem of linear MMSE estimation. As we mentioned, in the local case, to better track the nonstationary structure of the process, the kernel may depend on time and frequency, and the superscript $(t,\xi)$ denotes this possible dependence.

\noindent First, note that for each value of $(t,\xi)$, $\{P_X(t,\xi); \phi\in L^2(\mathbbm{R}\times \mathbbm{R}^+)\}$ belongs to a Hilbert space $\mathcal{H}$ of second-order random variables. The scale invariant ambiguity function  $A_X(\theta,\tau)$ generates a subspace $S_X$ of $\mathcal{H}$ as defined by $P_X(t,\xi)$ in (\ref{13'}), ref \cite{say}:
\begin{equation}
S_X=\{P_X(t,\xi): \phi \in L^2(\mathbbm{R}\times \mathbbm{R}^+) \}.\label{37}
\end{equation}
\noindent In the local problem (\ref{14'}), the orthogonal projection of $W_{E,X}(t,\xi)$ onto $\overline{S}_X$ is desired. So,
using the expression for $P_X(t,\xi)$ in (\ref{13'}) and by orthogonality principle, $P_X(t,\xi)$ achieves minimum mean-square error if and only if \cite{say}
\begin{equation}
E\big{\{}\big{[}P_X(t,\xi)-W_{E,X}(t,\xi)\big{]}A^*_X(\theta ',\tau ')\big{\}}=0,\quad \quad for \; all \; (\theta ',\tau '),\label{38}
\end{equation}
where $A^*_X$ is a complex conjugate of $A_X$.

\noindent Expressing $P_X$ as in (\ref{13'}), we can write (\ref{38}) as
$$\int_{-\infty}^\infty \int_0^\infty E\big{\{}A_X(\theta, \tau)A^*_X(\theta ',\tau ')\big{\}}\phi_{opt}^{(t,\xi)}(\theta, \tau)t^{i2\pi\theta}\tau^{-i2\pi\xi-1} d\tau d\theta$$
\begin{equation}
\qquad\qquad\qquad=W_{E,X}(t,\xi)E\big{\{}A^*_X(\theta ',\tau ')\big{\}}, \quad for \; all\; (\theta ',\tau ').\label{39}
\end{equation}
The above linear equation characterizes the locally optimal kernel $\phi_{opt}^{(t,\xi)}$. The linear equation is of the form
\begin{equation}
\mathcal{A}_{(t,\xi)}\phi=W_{E,X}(t,\xi) b,\label{kern}
\end{equation}
where $\mathcal{A}: L^2(\mathbbm{R} \times \mathbbm{R}^+) \rightarrow L^2(\mathbbm{R}^+ \times \mathbbm{R})$ is a linear operator, and $b=EA^*_X \in L^2(\mathbbm{R}\times \mathbbm{R}^+)$. Thus, the solution can be written as
\begin{equation}
\phi_{opt}^{(t,\xi)}=W_{E,X}(t,\xi)\mathcal{A}^{\dag}_{(t,\xi)} b, \label{40}
\end{equation}
where the superscript $\dag$ denotes the pseudo-inverse. By insertion of (\ref{kern}) into (\ref{13'}), we have that
$$P_X(t,\xi)=\int_{-\infty}^\infty \int_0^\infty A_X(\theta, \tau)\phi_{opt}^{(t,\xi)}(\theta, \tau)t^{i2\pi\theta}\tau^{-i2\pi\xi-1} d\tau d\theta$$
$$\qquad\quad \;\;\qquad\qquad \;=W_{E,X}(t,\xi)\int_{-\infty}^\infty \int_0^\infty A_X(\theta, \tau)\{\mathcal{A}^{\dag}_{(t,\xi)} b\} t^{i2\pi\theta}\tau^{-i2\pi\xi-1} d\tau d\theta$$
\begin{equation}
=W_{E,X}(t,\xi)\tilde{P_X}(t,\xi),\qquad \qquad\qquad\qquad \label{41}
\end{equation}
this means that, for each value of $(t,\xi)$, the optimal estimate of $W_{E,X}(t,\xi)$ from a realization of a scale invariant process $X$ is $W_{E,X}(t,\xi)$ itself, multiplied by a TFR of $x$ generated with the kernel $\mathcal{A}^{\dag}_{(t,\xi)} b$.

\subsubsection{The global estimation}
\noindent Now, if we consider $\phi_{opt}$ to be independent of $(t,\xi)$, then by multiplying both sides of (\ref{39}) with $t^{-i2\pi\theta '-1} (\tau ')^{i2\pi\xi}$, and integrating over $(t, \xi)$, we obtain the global optimal ambiguity-domain kernel as:

\begin{equation}
\phi_{opt}(\theta, \tau)=\frac{|A_{E,X}(\theta, \tau)|^2}{E|A_X(\theta, \tau)|^2}I_U(\theta, \tau),\label{15}
\end{equation}
where $I_U(\theta, \tau)$ denotes the indicator function for the open set $$U=\{(\theta, \tau);E|A_X(\theta, \tau)|^2>0\}\subset \mathbbm{R} \times \mathbbm{R}^+.$$

\begin{proof}
 See A2.
\end{proof}

The optimal time-frequency kernel is computed by
\begin{equation}
\Phi_{opt}(t,\xi)=\mathcal{M}_1^{-1}\mathcal{M}_2  \phi_{opt}(\theta, \tau).\label{16}
\end{equation}\\

For circularly symmetric Gaussian processes, using (\ref{9}) and (\ref{aex}) we obtain
\begin{equation}
E|A_X(\theta,\tau)|^2=|A_{E,X}(\theta,\tau)|^2+D_1(\theta,\tau),\label{18}
\end{equation}
where
\begin{equation}
D_1(\theta,\tau)=\int_0^\infty \int_0^\infty R_X(t_1\sqrt{\tau},t_2\sqrt{\tau})R^*_X(t_1/\sqrt{\tau},t_2/\sqrt{\tau})
t_1^{-i2\pi\theta-1}t_2^{i2\pi\theta-1}dt_1 dt_2.\label{19}
\end{equation}

In the case of real-valued processes, using (\ref{8}), we get
\begin{equation}
E|A_X(\theta,\tau)|^2=|A_{E,X}(\theta,\tau)|^2+D_1(\theta,\tau)+D_2(\theta,\tau)\label{20}
\end{equation}
where
\begin{equation}
D_2(\theta,\tau)=\int_0^\infty \int_0^\infty R_X(t_1\sqrt{\tau},t_2/\sqrt{\tau})R_X(t_1/\sqrt{\tau},t_2\sqrt{\tau})
t_1^{-i2\pi\theta-1}t_2^{i2\pi\theta-1}dt_1 dt_2.\label{D2}
\end{equation}
For validity of (\ref{18}) and (\ref{20}), see A3.\\

\begin{pro}
The optimal kernel for a circularly symmetric Gaussian LSSP, is
\begin{equation}
\phi_{opt}(\theta, \tau)=\frac{|C_X(\tau)|^2 |(\mathcal{M}Q)(i2\pi\theta)|^2}{|C_X(\tau)|^2|(\mathcal{M}Q)(i2\pi\theta)|^2 +(\mathcal{M}|C_X|^2)(i2\pi\theta)\;(Q\varoast Q^*)(\tau)},\label{22}
\end{equation}
\noindent where $(Q\varoast Q^*)(\tau):=\int_0^{\infty}Q(u\sqrt{\tau})Q^*(u/\sqrt{\tau})\frac{du}{u}$. To obtain the optimal kernel for a real-valued Gaussian LSSP, we have to take into account the third term $D_2(\theta, \tau)$ into the denominator. It is
\begin{equation}
D_2(\theta, \tau)= A_{C_X}(\theta, \tau^2)\int_0^\infty Q^2(u)\frac{du}{u},\label{d2}
\end{equation}
\noindent where, $A_{C_X}(\theta, \tau^2):=\int_0^\infty C_X(v\tau)C_X(v/\tau) v^{-i2\pi\theta-1} dv$.
\end{pro}
\begin{proof}
See A4.
\end{proof}

\section{The Optimal Kernel for Some Generalizations of LSSPs}
In order to generalize the LSSP definition, we introduce two classes of locally self-similar chrip process, and Multicomponent Locally Self-Similar process.

\subsection{Locally Self-Similar Circularly Symmetric Gaussian Chrip Process}
\begin{defi}
A locally self-similar chrip process (LSSCP) is a Gaussian circularly symmetric process with covariance function of the form
\begin{equation}
R_Z(t,s)=Q(\sqrt{ts})C_X(t/s)l_{a,b}(t,s), \label{24}
\end{equation}
where $l_{a,b}$ is defined by
\begin{equation}
l_{a,b}(t,s)=(t/s)^{i a(\ln \sqrt{ts}-b)},\quad \quad a,b\in \mathbbm{R}. \label{25}
\end{equation}
\end{defi}

\noindent The constant $a$ determines the chrip frequency and $b$ the start of the chrip frequency.
\par
If $X(t)$ is a LSSP with covariance $Q\varotimes C_X\circ k^{-1}$, then $Z(t):=X(t)t^{ia(\ln\sqrt{t}-b)}$ i.e., $Z$ equals $X$ times a chrip, has covariance $R_Z$. We have $R_Z(t\sqrt{\tau},t/\sqrt{\tau})=Q(t)C_X(\tau)\tau^{ia(\ln t-b)}$, which shows that $C(\tau)$ is modulated by the chrip $\tau^{ia(\ln t-b)}$.\\

If one multiplies a signal with a chrip according to $f(t):=f_0(t)t^{i(a/2)\ln t}, a\in \mathbbm{R}$, then the SIWD changes to $W(t, \xi)\rightarrow W(t, \xi-\frac{a}{2\pi}\ln t)$:

$$W_f(t,\xi)=\int_0^\infty f(t\sqrt{\tau})f^*(t/\sqrt{\tau})\tau^{-i2\pi\xi-1} d\tau$$

\input{epsf}
\epsfxsize=2.5in \epsfysize=1.5in
\begin{figure}
\vspace{-2.6in}
\centerline{\epsfxsize=8in \epsfysize=6in \epsffile{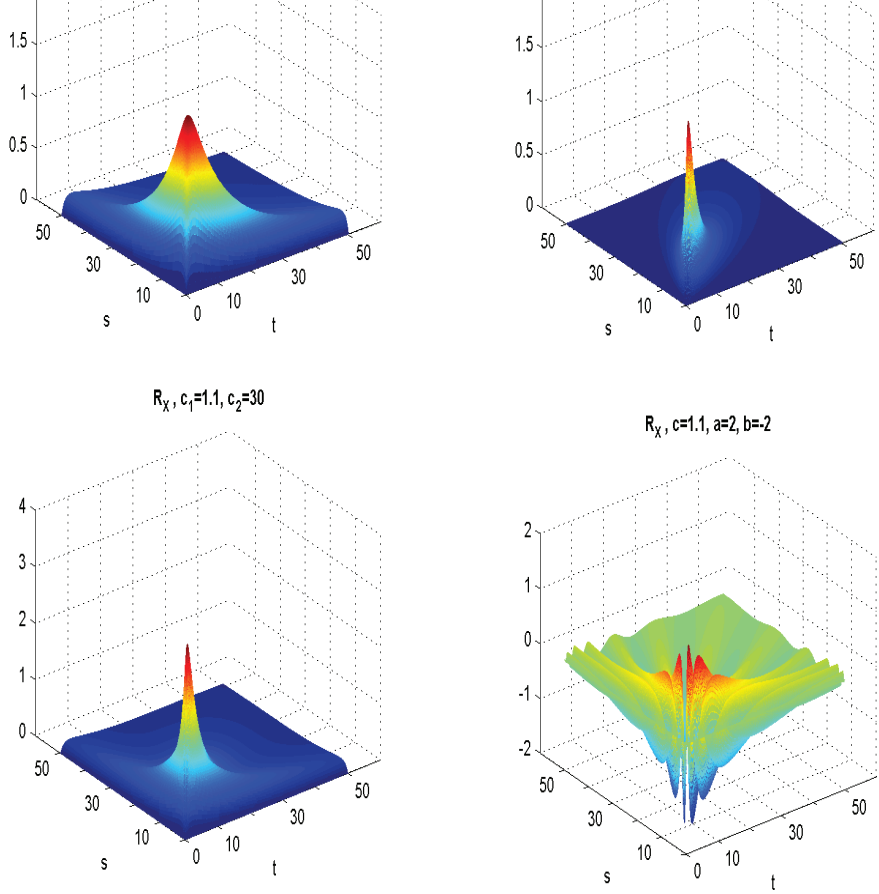}}
\vspace{-0.3in}
\caption{\scriptsize Example of covariance matrices $\bf{R}_X$ of LSSPs with $H=0.5$ and different parameter values. Top left: LSSP: $c=1.1$. Top right: LSSP, $c=30$. Bottom left: MLSSP, $c_1=1.1, c_2=30$. Bottom right: LSSCP, $c=1.1$, $a=2$, $b=-2$. }
\end{figure}

\vspace{2.5in}
\input{epsf}
\epsfxsize=2.5in \epsfysize=1.5in
 \begin{figure}
\vspace{-3in}
\centerline{\epsfxsize=9in \epsfysize=7in \epsffile{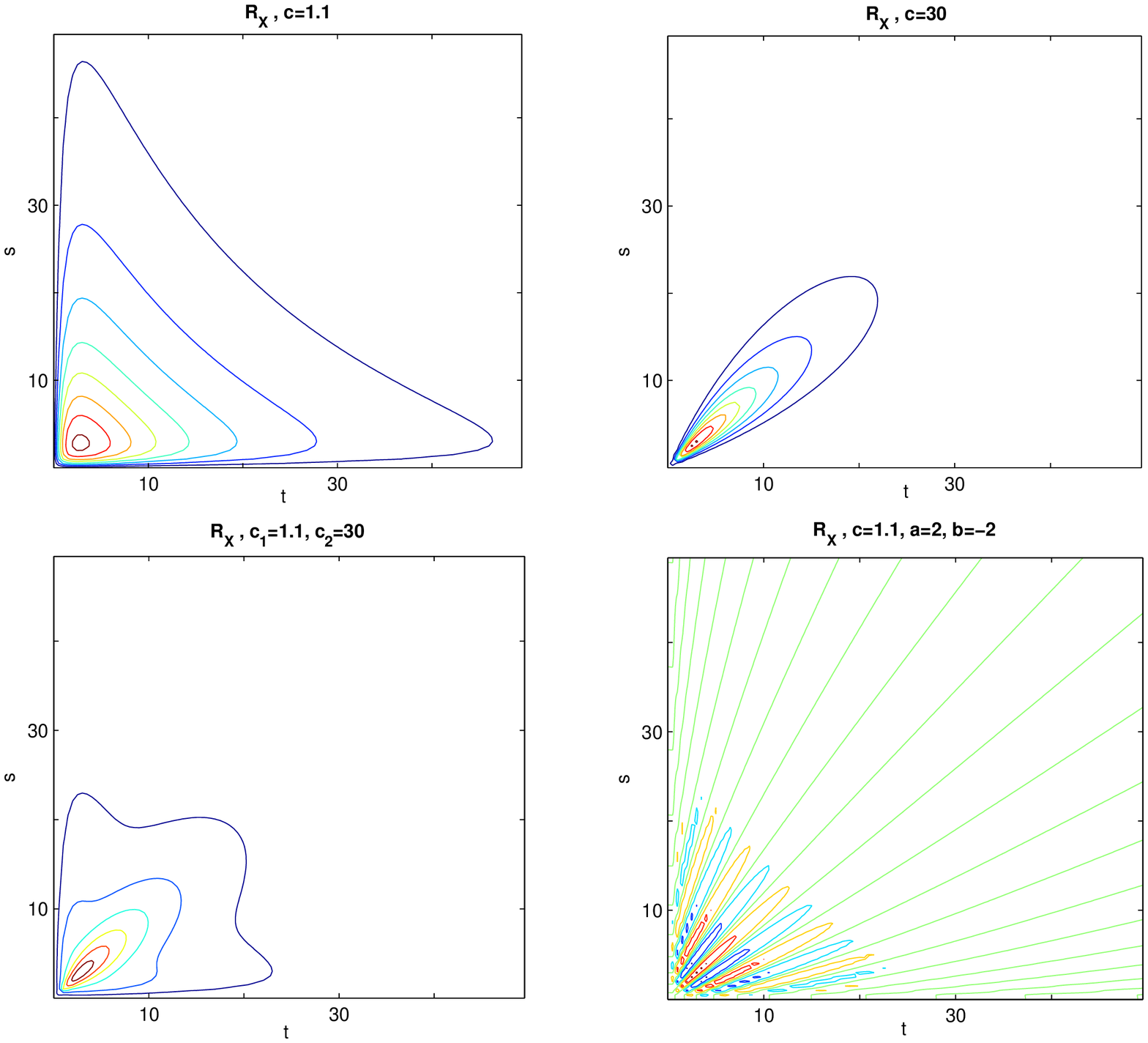}}
\vspace{-0.1in}
\caption{\scriptsize Contour plots of covariance matrices $\bf{R}_X$ of LSSPs for $H=0.5$ and different parameter values (corresponding to Figure 1). Top left: LSSP, $c=1.1$. Top right: LSSP, $c=30$. Bottom left: MLLSP, $c_1=1.1$, $c_2=30$. Bottom right: LSSCP, $c=1.1$, $a=2$, $b=-2$.}
\end{figure}

\vspace{0.5in}
\input{epsf}
\epsfxsize=2.5in \epsfysize=1.5in
 \begin{figure}
\vspace{-2.9in}

\centerline{\epsfxsize=8in \epsfysize=6in \epsffile{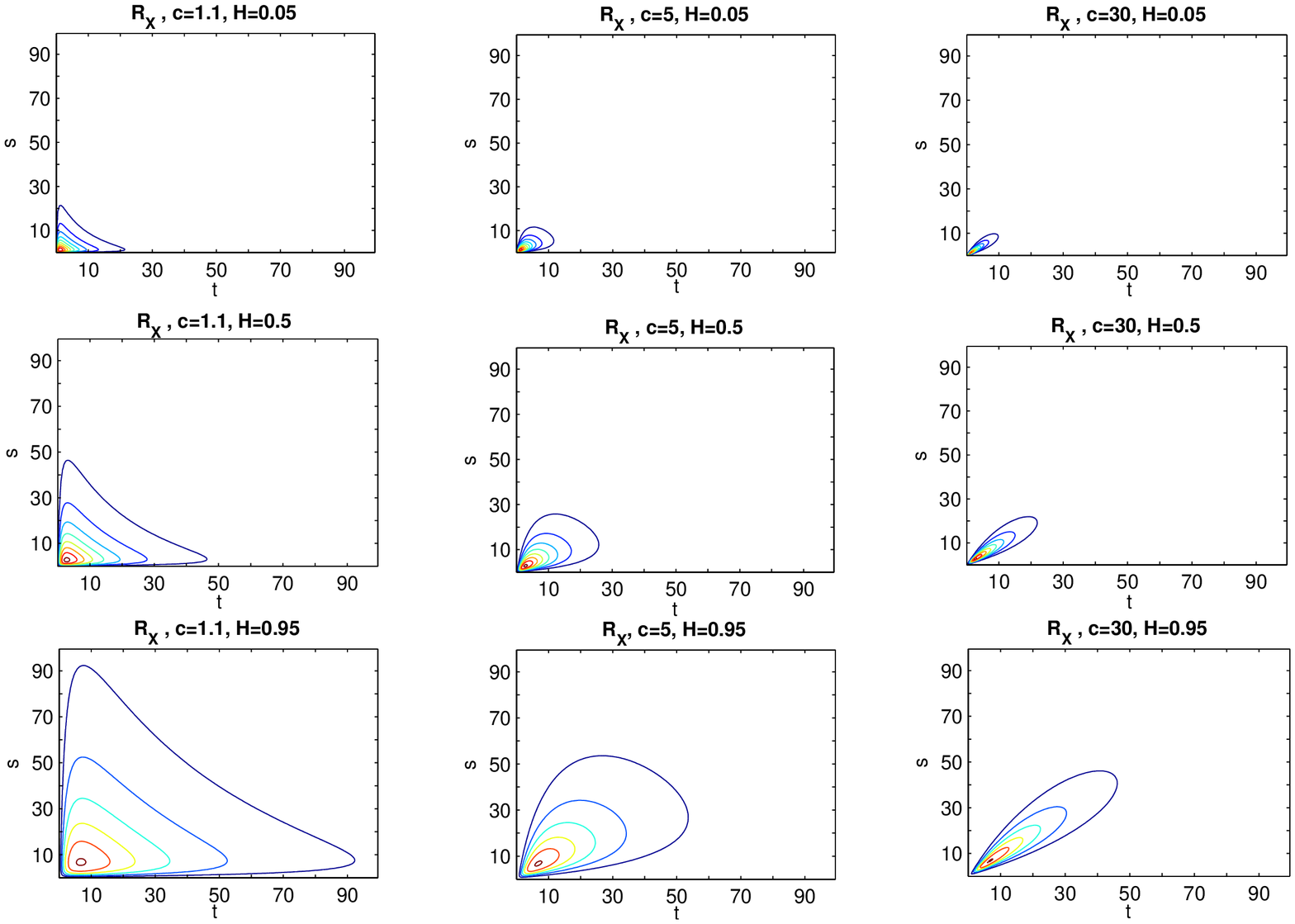}}
\vspace{0.2in}
\caption{\scriptsize Contour plot of covariance matrices $\bf{R}_X$ of LSSPs for different parameter values $c$ and $H$. Left column: LSSP, $c=1.1$, and $H$ ranging from (top) $H=0.05$ to (bottom) $H=0.95$. Middle column: LSSP, $c=5$, and $H$ varies from $H=0.05$ to $H=0.95$. Right column: LSSP, $c=30$. Each column shows that the covariance shape for a constant $c$ is identical, but as  $H$ increases, the slope of $R_X$ decreases. Also, each row shows that, with the same Hurst parameter $H$, as the parameter $c$ increases, the peaks in sides are inclined into the peak in the middle.}
\end{figure}

$$\qquad \qquad\qquad \qquad=\int_0^\infty f_0(t\sqrt{\tau})f_0^*(t/\sqrt{\tau}) \tau^{-i2\pi(\xi-\frac{a}{2\pi} \ln t)-1} d\tau\;\;$$

$\qquad\qquad \qquad \qquad \qquad\quad \quad\;\;\;\;\;\;=W_{f_0}(t,\xi-\frac{a}{2\pi}\ln t).$\\

The next theorem states that the optimal kernel for an LSSCP is obtained by  a transformation of the kernel of the corresponding LSSP.

\begin{theo}
Suppose $C_X$ and $Q$ define a circular symmetric Gaussian LSSP with covariance given by (\ref{rxk}). Then, the optimal time-frequency kernel for a LSSCP process with chrip parameters $a$ and $b$, denoted $\Phi_{c.opt}$, is
\begin{equation}
\Phi_{c.opt}(t,\xi)=\Phi_{opt}(t,\xi-\frac{a}{2\pi}\ln t)\label{26}
\end{equation}
where $\Phi_{opt}$ is the optimal kernel of the corresponding LSSP (i.e., $a=b=0$), defined by (\ref{16}) and  (\ref{22}).
\end{theo}
\begin{proof}
See A5.
\end{proof}

\subsection{Multicomponent Locally Self-Similar Circularly Symmetric Gaussian Process}

\begin{defi}
A multicomponent LSSP (MLSSP) is a process whose covariance has the form
\begin{equation}
R_X(t,s)=\sum_{j=1}^\infty Q_j(\sqrt{ts}) C_{X_j}(t/s),\label{27}
\end{equation}
where each term $ Q_j\varotimes C_{X_j}\circ k^{-1}$ is the covariance function of a LSSP.
\end{defi}

By insertion of (\ref{27}) into (\ref{wvs}) and (\ref{aex}), the SIWS and ESIAF are
$$W_{E,X}(t,\xi)=\int_0^\infty \sum_{j=1}^\infty Q_j(t) C_{X_j}(\tau) \tau^{-i2\pi\xi-1} d\tau$$
$$\;\;\;\;\;=\sum_{j=1}^\infty Q_j(t) (\mathcal{M}C_{X_j})(i2\pi\xi),$$

$$A_{E,X}(\theta, \tau)=\int_0^\infty \sum_{j=1}^\infty Q_j(t) C_{X_j}(\tau) t^{-i2\pi\theta-1}dt$$
$$\quad\quad=\sum_{j=1}^\infty  C_{X_j}(\tau) (\mathcal{M}Q_j)(i2\pi\theta).$$

\vspace{2.5in}
\input{epsf}
\epsfxsize=2.5in \epsfysize=1.5in
\begin{figure}
\vspace{-1.7 in}
\centerline{\epsfxsize=8in \epsfysize=5in \epsffile{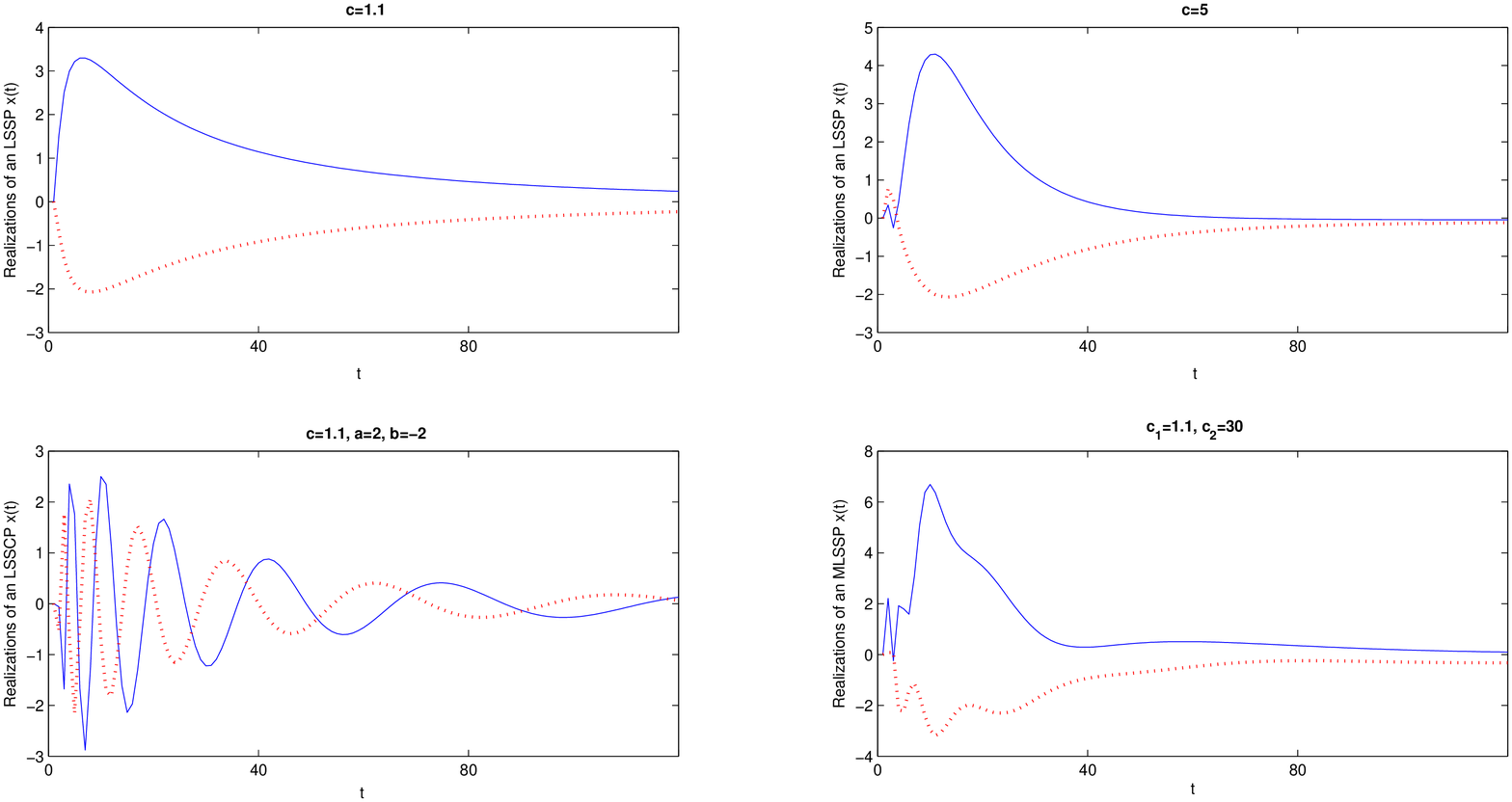}}
\vspace{-0.1in}
\caption{\scriptsize Sample realizations of LSSPs with $H=0.5$ corresponding to the covariance function $R_X(t,s)$ in Figure 1. Top left: LSSP, $c=1.1$. Top right: LSSP, $c=5$.  Bottom left: LSSCP, $c=1.1$, $a=2$, $b=-2$ (real-valued part only), and Bottom right: MLSSP, $c_1=1.1$, $c_2=30$, which is a combination of LSSPs with $c=1.1$ and $c=30$.}
\end{figure}

\vspace{2.5in}
\input{epsf}
\epsfxsize=3in \epsfysize=2in
 \begin{figure}
\vspace{-2in}
\centerline{\epsfxsize=7in \epsfysize=5in \epsffile{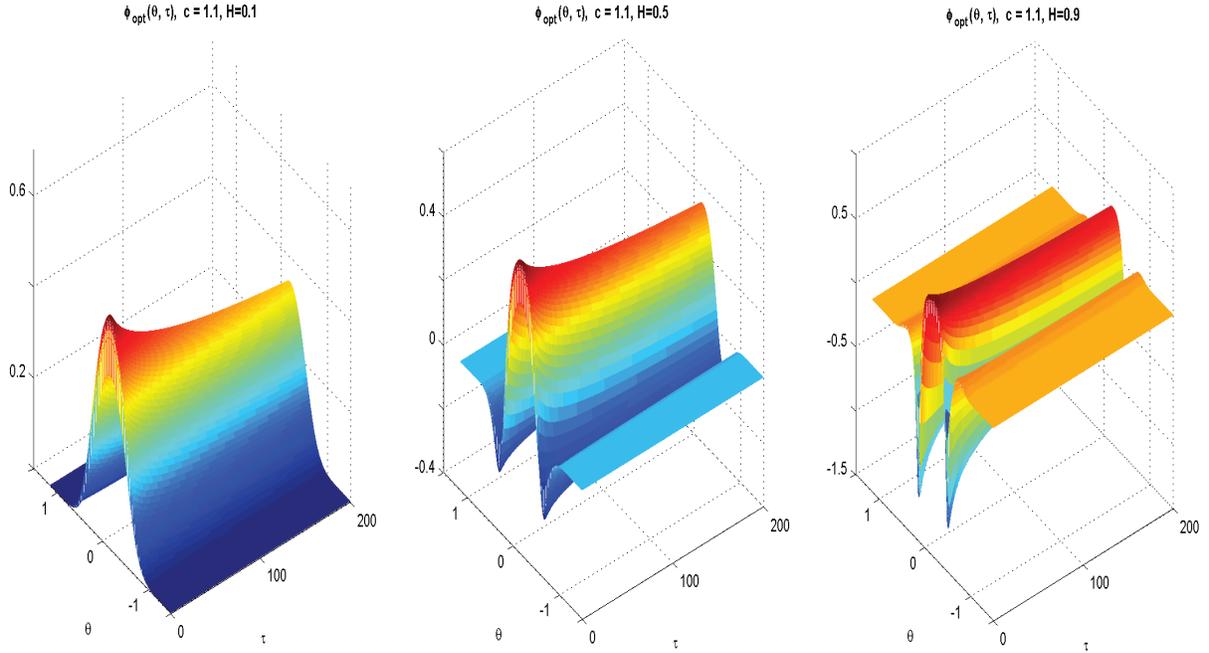}}
\vspace{-0.3in}
\caption{\scriptsize The mean-square error optimal ambiguity domain kernel corresponding to the LSSP with covariance function in Figure 1, with $c=1.1$
  Left: $\phi_{opt}^{LSSP}(\theta, \tau)$, $H = 0.1$. Middle: $\phi_{opt}^{LSSP}(\theta, \tau)$, $H = 0.5$. Right: $\phi_{opt}^{LSSP}(\theta, \tau)$, $H = 0.9$.}
\end{figure}

Restricting to the circularly symmetric case; and using (\ref{18}) and (\ref{19}), the optimal ambiguity domain kernel for circularly symmetric MLSSP is derived as
\begin{equation}
\phi(\theta, \tau)=\frac{|\sum_{j=1}^\infty  C_{X_j}(\tau) (\mathcal{M}Q_j)(i2\pi\theta)|^2}{|\sum_{j=1}^\infty  C_{X_j}(\tau) (\mathcal{M}Q_j)(i2\pi\theta)|^2+\sum_{j,k=1}^\infty \mathcal{M}(C_{X_j} {C}^*_{X_k})(i2\pi\theta)\; (Q_j\varoast Q^*_k)(\tau)},\label{mlsscpkernel}
\end{equation}

\noindent where the second term of the denominator is
$$D_1(\theta, \tau)=\int_0^\infty\int_0^\infty\sum_{j,k=1}^\infty Q_j(\sqrt{t_1t_2\tau})Q^*_k(\sqrt{\frac{t_1t_2}{\tau}})C_{X_j}(\frac{t_1}{t_2})C^*_{X_k}(\frac{t_1}{t_2}) t_1^{-i2\pi\theta-1} t_2^{i2\pi\theta-1} dt_1 dt_2$$

$$=\sum_{j,k=1}^\infty \int_0^\infty C_{X_j}(v)C^*_{X_k}(v) v^{-i2\pi \theta-1}dv \int_0^\infty Q_j(u\sqrt{\tau})Q^*_k(u/\sqrt{\tau}) \frac{du}{u}.$$

As a generalization of LSSCPs and MLSSPs, we can introduce a multicomponent locally self-similar chrip process (MLSSCP) whose covariance function has the form
\begin{equation}
R_X(t,s)= \sum_{j=1}^\infty Q_j(\sqrt{ts}) C_{X_j}(t/s) (t/s)^{i a_j(\ln \sqrt{ts}-b_j)}, \label{29}
\end{equation}
\noindent where each term are LSSCP covariances with individual constants $a_j$ and $b_j$. A MLSSCP has a covariance that is a sum of covariances with chrip behavior of various localization in time and chrip constants.

By insertion of (\ref{29}) into (\ref{aex}) and (\ref{19}), we have

$$A_{E,X}(\theta, \tau)=\int_0^\infty \sum_{j=1}^\infty Q_j(t) C_{X_j}(\tau) \tau^{ia_j(\ln t-b_j)}t^{-i2\pi\theta-1} dt$$

$$\quad\quad\quad\;\;\;\;\;\;\;\;\;\;\;=\sum_{j=1}^\infty C_{X_j}(\tau) \tau^{-ia_jb_j}\int_0^\infty Q_j(t)t^{-i2\pi\theta+ia_j\ln \tau-1} dt$$

\begin{equation}
\quad\quad\;\;\;\;\;\;\;\;\;\;\;=\sum_{j=1}^\infty C_{X_j}(\tau) \tau^{-ia_jb_j} (\mathcal{M} Q_j)(i2\pi(\theta-\frac{a_j}{2\pi} \ln \tau)),\label{mlsscpa}
\end{equation}

$$D_1(\theta, \tau)=\int_0^\infty\int_0^\infty\sum_{j,k=1}^\infty Q_j(\sqrt{t_1t_2\tau})Q^*_k(\sqrt{\frac{t_1t_2}{\tau}})C_{X_j}(\frac{t_1}{t_2})C^*_{X_k}(\frac{t_1}{t_2})$$

$$\;\;\;\;\;\;\;\;\;\;\quad\times(\frac{t_1}{t_2})^{ia_j(\ln \sqrt{t_1t_2\tau}-b_j)-ia_k(\ln \sqrt{t_1t_2/\tau}-b_k)}\; t_1^{-i2\pi\theta-1} t_2^{i2\pi\theta-1} dt_1 dt_2$$

$$\;\;\;\;=\sum_{j,k=1}^\infty \int_0^\infty\int_0^\infty Q_j(u\sqrt{\tau})Q^*_k(u/\sqrt{\tau})C_{X_j}(v)C^*_{X_k}(v)$$
$$\times v^{-i2\pi\{\theta-\frac{ a_j}{2\pi}(\ln u\sqrt{\tau}-b_j)+\frac{ a_k}{2\pi}(\ln \frac{ u}{\sqrt{\tau}}-b_k)\}-1}\;\frac{du}{u}\;dv.$$
Thus, we have that
$$D_1(\theta, \tau)=\sum_{j=1}^\infty (\mathcal{M}|C_{X_j}|^2)(i2\pi(\theta-\frac{ a_j}{2\pi}\ln \tau)) \int_0^\infty Q_j(u\sqrt{\tau})Q^*_j(u/\sqrt{\tau})\frac{du}{u}$$

\begin{equation}
+\sum_{j,k=1, j\neq k}^\infty\int_0^\infty Q_j(u\sqrt{\tau})Q^*_k(u/\sqrt{\tau}) g(\theta,u)\frac{du}{u},\label{mlsscp}
\end{equation}
where
$$g(\theta,u)=\mathcal{M}(C_{X_j}(v)C^*_{X_k})(\theta-\frac{a_j}{2\pi}(\ln u\sqrt{\tau}-b_j)+\frac{a_k}{2\pi}(\ln \frac{u}{\sqrt{\tau}}-b_k)).$$\\

By insertion of (\ref{mlsscpa}) and (\ref{mlsscp}) into (\ref{15}), the optimal ambiguity domain kernel for a circularly symmetric Gaussian MLSSCP is obtained.

\subsection{Discrete-time Examples}
Here we study an example of a circularly symmetric Gaussian LSSP, LSSCP and MLSSP  with covariance functions (\ref{rxk}), (\ref{24}) and (\ref{27})  respectively. We consider
\begin{equation}
Q(\tau)=\tau^{2H-\frac{1}{2}\ln \tau}, \;\qquad C_X(\tau)=\tau^{-\frac{c}{8}\ln \tau}, c\geq 1.\label{example}
\end{equation}
 For a given $Q$ and $C_X$, the covariance matrices, $R_X$, of the processes are computed, and depicted in Figure 1, for $H=0.5$, $c=1.1$, $c=30$; and the corresponding contour plots are shown in Figure 2. Also, the contour plots of the LSSP for three different parameter values of $H$ ranging from $H=0.05$ to $H=0.95$,  and $c=1.1$, $c=5$ and $c=30$  are displayed in Figure 3.

The sample paths of  the circularly symmetric Gaussian LSSP, LSSCP and MLSSP, corresponding to the covariance matrices of Figure 1, are simulated and displayed in Figure 4 with $H=0.5$. The realizations are simulated from the covariance matrices $R_X$, according to
\begin{equation}
 X=Lu,
\end{equation}
where $u$ is a realization of a real-values white Gaussian zero mean stochastic process with variance one, and the matrix $L$ is related to the covariance matrix $R_X$ as
\begin{equation}
R_X(t,s)=E(X(t)X^*(s))=L\;E(u(t)u^*(s))\;L^*=LL^*.
\end{equation}

The optimal ambiguity domain kernel for a circularly symmetric Gaussian LSSP, LSSCP and MLSSP are obtained in (\ref{22}), (\ref{lsscp}) and (\ref{mlsscpkernel}), where for given $Q$ and $C_X$ in (\ref{example}) are computed as

\begin{equation}
\phi_{opt}^{LSSP}(\theta,\tau)=\frac{1}{1+c^{-1/2}\; e^{(1-1/c)(2\pi\theta)^2+4 H i(2\pi\theta)}\; \tau^{(\frac{c-1}{4})\ln\tau}},\label{phi1}
\end{equation}

\begin{equation}
\phi_{opt}^{LSSCP}(\theta,\tau)=\frac{1}{1+c^{-1/2}\; e^{(1-1/c)(2\pi\theta-a\ln \tau)^2+4Hi(2\pi\theta-a\ln \tau)}\; \tau^{(\frac{c-1}{4})\ln\tau}},\label{phi2}
\end{equation}

\begin{equation}
\phi_{opt}^{MLSSP}(\theta,\tau)=\bigg{(}1\big{/}\bigg{(} 1+\frac{\sum_{j,k}\sqrt{\frac{2}{c_j+c_k}}\;e^{-\frac{2}{c_j+c_k}(2\pi\theta)^2+(H_j+H_k)^2} \; \tau^{(H_j-H_k)-\frac{1}{4} \ln \tau}}{\big{|}\sum_{j} e^{\frac{1}{2}(2H_j-i2\pi\theta)^2} \tau^{-\frac{c_j}{8}\ln \tau} \big{|}^2  } \bigg{)}\bigg{)},\label{phi3}
\end{equation}

\begin{equation}
\phi_{opt}^{MLSSCP}(\theta,\tau)=\bigg{(}1\big{/}\bigg{(} 1+\frac{\sum_{j,k}\sqrt{\frac{2}{c_j+c_k}}\;e^{-\frac{2}{c_j+c_k}(2\pi\theta-a\ln \tau)^2+(H_j+H_k)^2} \; \tau^{(H_j-H_k)-\frac{1}{4} \ln \tau}}{\big{|}\sum_{j} e^{\frac{1}{2}(2H_j-i(2\pi\theta-a\ln \tau))^2} \tau^{-\frac{c_j}{8}\ln \tau} \big{|}^2  } \bigg{)}\bigg{)},\label{phi4}
\end{equation}

Eq (\ref{phi4}) is the optimal kernel of a MLSSCP which is derived as a combination of the optimal kernels for LSSCP and MLSSP. For validity of (\ref{phi1})-(\ref{phi4}), see A6. The optimal ambiguity domain kernels of the LSSP, Eq (\ref{phi1}), for $c=1.1$ and three different Hurst parameters ranging from $H=0.1$ to $H=0.9$ are plotted in Figure 5.

\input{epsf}
\epsfxsize=3in \epsfysize=2in
 \begin{figure}
\vspace{-1.8in}
\centerline{\epsfxsize=7.5in \epsfysize=4.5in \epsffile{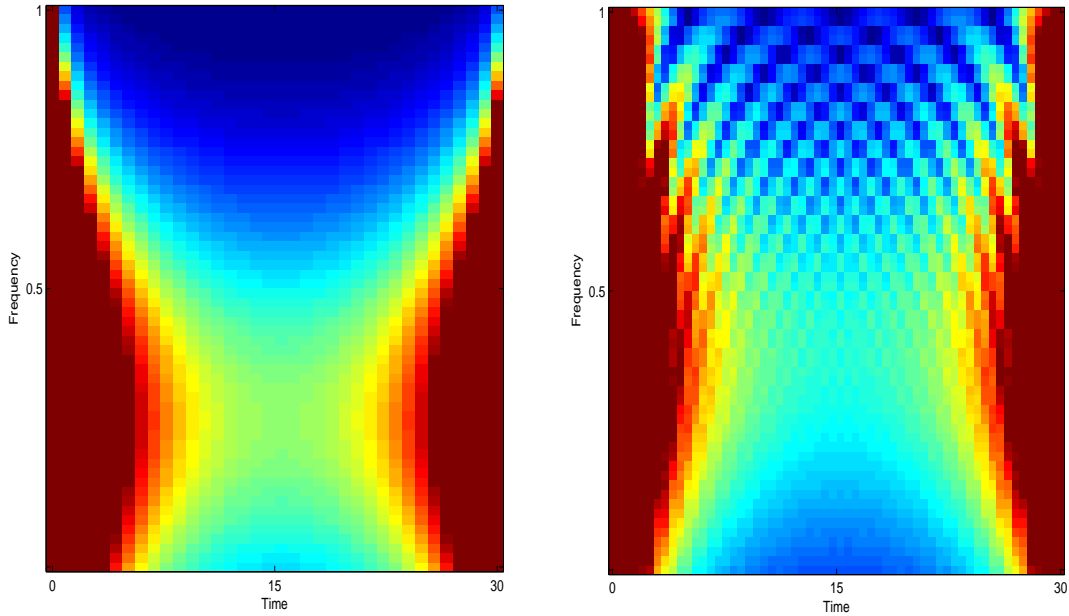}}
\vspace{-0.3in}
\caption{\scriptsize Left: True SIWS of the Gaussian LSSP. Right: The SIWD of the Gaussian LSSP.}
\end{figure}

\subsection{Simulation}
Now, the predominance of the proposed method for the SIWS estimation of a Gaussian LSSP is investigated over the classical optimal kernel WVS estimation. To this end, the performance of the globally optimal kernel estimator is compared with that of the classical WVS.

Let $X$ be a Gaussian LSSP with the covariance matrix (\ref{rxk}) where $Q$ and $C_X$ are identified by (\ref{example}). We assume that $H=0.5$ and $c=1.1$, so
$$R_X(t,s)=\sqrt{ts}^{2H-\frac{1}{2}\ln {\sqrt{ts}}}(t/s)^{-\frac{c}{8}\ln {(t/s)}}.$$
The SIWS of a LSSP is given by (\ref{welssp}), and is computed for given $Q$ and $C_X$ , Figure 6 (Left). The SIWS is estimated from the realization $X$, using relation (\ref{13'}), where the optimal kernel $\phi(\theta, \tau)$ is computed in (\ref{phi1}), Figure 7 (Left). Clearly, the optimal kernel SIWS estimation gives a much more accurate estimate for the true SIWS, than the classical optimal kernel WVS estimation Figure 7 (Right).

The simulation results show that, although LSSPs are a subclass of nonstationary processes, but the scale invariant property exists in such processes, makes them different from the other nonstationary ones; and ordinary nonstationary spectrums and estimation methods may not be applicable for TF analysis of such processes. So, some special tools should be considered, which are compatible with the scale invariant property.

\vspace{2.5in}
\input{epsf}
\epsfxsize=3in \epsfysize=2in
 \begin{figure}
\vspace{-1.8in}
\centerline{\epsfxsize=7.5in \epsfysize=4.5in \epsffile{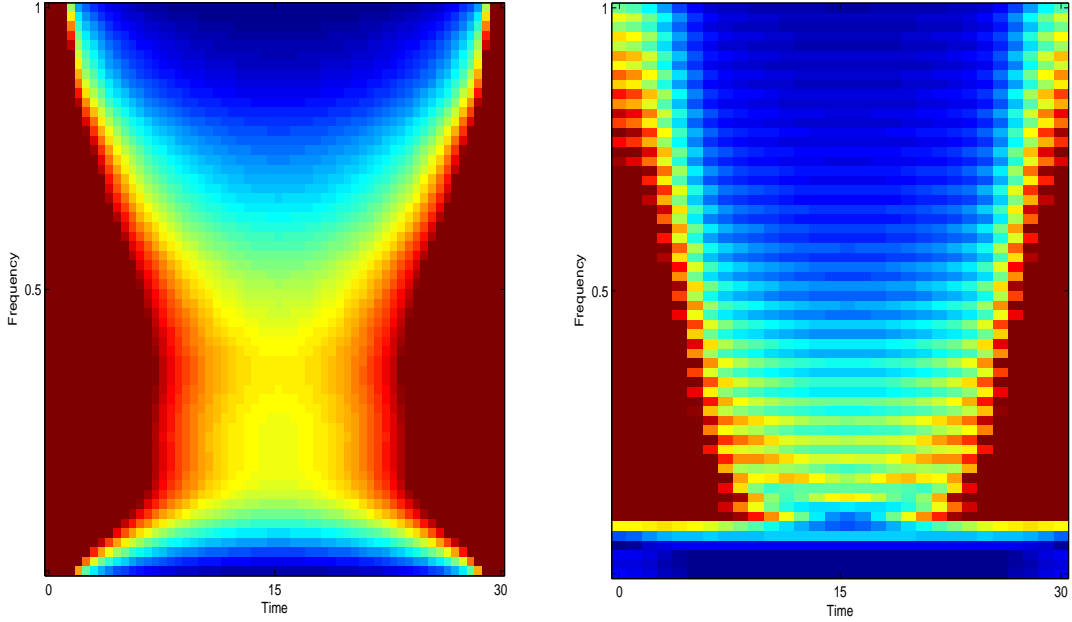}}
\vspace{-0.3in}
\caption{\scriptsize  Left: The optimal kernel SIWS estimation of a Gaussian LSSP. Right: The classical optimal WVS estimation.}
\end{figure}

\section{Conclusions}
In time-frequency analysis of locally self-similar processes (LSSPs), we obtained the scale invariant Wigner spectrum (SIWS) estimation by using the Cohen's class counterpart of time-frequency representations, which is compatible with the scale invariant property. By this, we provided a better estimation than the classical Wigner-Ville spectrum. By introducing the optimal kernel for the SIWS, which minimizes the mean-square error, our estimation has been modified and provided a close estimation to the true SIWS.

Although we restrict ourselves to Gaussian LSSPs, but the MMSE optimal kernel can be computed for the SIWS estimation of non-Gaussian LSSPs. In such problems, the optimal kernel would be as a linear combination of the optimal kernels of Gaussian LSSPs.


\section{Appendix}
\textbf{A1:}
\noindent Proof of proposition $6$:
If $X(t)$ is a multiplicative harmonizable process, then from (\ref{2'})
we have that:
$$R_X(t\sqrt{\tau},t/\sqrt{\tau})=\int_{-\infty}^\infty \int_{-\infty}^\infty (t\sqrt{\tau})^{H+i2\pi\beta} (t/\sqrt{\tau})^{H-i2\pi\sigma} m(\beta,\sigma)  d\beta d\sigma.$$
By insertion of the above relation into the definition of  the SIWS , Eq (\ref{wvs}), $$W_{E,X}(t,\xi)= \int_{0}^\infty \int_{-\infty}^\infty \int_{-\infty}^\infty t^{2H+i2\pi(\beta-\sigma)} \tau^{i2\pi(\frac{\beta+\sigma}{2})} \tau^{-i2\pi\xi-1} m(\beta,\sigma)  d\beta d\sigma d\tau$$

$$\qquad\qquad\qquad\quad=\int_0^\infty \int_{-\infty}^\infty \int_{-\infty}^\infty  t^{2H+i2\pi\theta}  \tau^{-i2\pi(\xi-f)-1} m(f+\theta/2,f-\theta/2) d\theta df d\tau,$$

\noindent where $\theta:=\beta-\sigma$ and $f:=\frac{\beta+\sigma}{2}$.
Using relation $\int_0^\infty \tau^{-i2\pi(\xi-f)-1}d\tau=\delta(\xi-f)$, where $\delta(.)$ is a dirac delta function, (\ref{wteta}) is achieved.

$\hspace{6in} \boxempty$\\

\noindent\textbf{A2:}
Validity of Eq (\ref{15}): By multiplying both sides of (\ref{39}) with $t^{-i2\pi\theta '-1} (\tau ')^{i2\pi\xi}$ and integrating over $(t,\xi)$, the l.h.s leads to:
$$\int_{-\infty}^\infty \int_0^\infty\int_{-\infty}^\infty \int_0^\infty E\big{\{}A_X(\theta, \tau)A^*_X(\theta ',\tau ')\big{\}}\phi_{opt}(\theta, \tau) t^{-i2\pi(\theta '-\theta)-1} (\frac{\tau}{\tau '})^ {-i2\pi\xi} \;\frac{d\tau}{\tau} d\theta dt d\xi$$

$$\qquad\qquad\;\;\;= \int_{-\infty}^\infty \int_0^\infty E \big{\{}A_X(\theta, \tau)A^*_X(\theta ',\tau ')\big{\}}\phi_{opt}(\theta, \tau)\delta(\theta '-\theta)\delta(\ln \tau - \ln \tau ')\frac{d\tau}{\tau} d\theta$$
$$= E|A_X(\theta, \tau)|^2\phi_{opt}(\theta, \tau),\qquad\qquad\;\;\;\qquad\qquad\;\;\;\qquad\qquad\;\;\;\;\;\;$$

\noindent where $\int_0^\infty  t^{-i2\pi u-1} dt=\delta(u)$ and  $\int_{-\infty}^\infty v^{-i2\pi\xi}d\xi=\delta(\ln v)$. By insertion of $W_{E,X}=\mathcal{M}_1^{-1}\mathcal{M}_2 A_{E,X}$ into the r.h.s. of (\ref{39}), we have that:

$$\int_{-\infty}^\infty \int_0^\infty\int_{-\infty}^\infty \int_0^\infty A_{E,X}(\theta,\tau)E\{A^*_X(\theta ',\tau ')\}t^{-i2\pi(\theta '-\theta)-1} (\frac{\tau}{\tau '})^ {-i2\pi\xi}\; \frac{d\tau}{\tau} d\theta dt d\xi=|A_{E,X}(\theta,\tau)|^2.$$

$\hspace{6in} \boxempty$\\

\noindent\textbf{A3:}
\noindent Validity of (\ref{18}) and (\ref{20}):
By (\ref{ax}),
$$E|A_X(\theta,\tau)|^2=E|A_X(\theta,\tau){A^*_X(\theta,\tau)}|$$
$$=E \bigg{\{}\int_0^\infty \int_0^\infty X(t_1\sqrt{\tau})
{X^*(t_1/\sqrt{\tau})} {X^*(t_2\sqrt{\tau})} X(t_2/\sqrt{\tau})t_1^{-i2\pi\theta-1}t_2^{i2\pi\theta-1}dt_1 dt_2 \bigg{\}}$$

$$=\int_0^\infty \int_0^\infty E\{X(t_1\sqrt{\tau}){X^*(t_1/\sqrt{\tau})}{X^*(t_2\sqrt{\tau})} X(t_2/\sqrt{\tau})
\}t_1^{-i2\pi\theta-1}t_2^{i2\pi\theta-1}dt_1 dt_2.$$
The integral inside the expectation operator is a stochastic integral. The interchange of expectation and integration is justified if the above-mentioned stochastic integral exists in m.s. sense \cite{say}. For a circularly symmetric Gaussian process, using (\ref{9}), we have that
$$E|A_X(\theta,\tau)|^2=\int_0^\infty \int_0^\infty R_X(t_1\sqrt{\tau},t_1/\sqrt{\tau}){R^*_X(t_2\sqrt{\tau},t_2/\sqrt{\tau})}
t_1^{-i2\pi\theta-1}t_2^{i2\pi\theta-1}dt_1 dt_2$$
\begin{equation}
\qquad\qquad\qquad\qquad\;\;\;\;+\int_0^\infty \int_0^\infty R_X(t_1\sqrt{\tau},t_2\sqrt{\tau}){R^*_X(t_1/\sqrt{\tau},t_2/\sqrt{\tau})}
t_1^{-i2\pi\theta-1}t_2^{i2\pi\theta-1}dt_1 dt_2.\label{a11}
\end{equation}

\noindent By the ESIAF, (\ref{aex}), the first integral is $|A_{E,X}(\theta,\tau)|^2$, and we define the latter one as $D_1(\theta,\tau)$. For the case of real-valued processes, using (\ref{8}), we have an extra term (\ref{D2}) in $E|A_X(\theta, \tau)|^2$.

$\hspace{6in} \boxempty$\\

\noindent\textbf{A4:} Proof of Proposition $7$:
For a circularly symmetric Gaussian processes, the denominator of (\ref{22}) is obtained from (\ref{18}). By (\ref{aex1}), $A_{E,X}(\theta,\tau)=C_X(\tau) (\mathcal{M}Q)(i2\pi\theta)$; and $D_1(\theta, \tau)$ is obtained from (\ref{19}) by $R_X(t,s)=Q(\sqrt{ts}) C_X(t/s)$ as

$$D_1(\theta,\tau)=\int_{0}^\infty\int_{0}^\infty Q(\sqrt{t_1t_2\tau}) C_X(t_1/t_2)Q^*(\sqrt{\frac{t_1t_2}{\tau}})C^*_X(t_1/t_2) t_1^{-i2\pi\theta-1}\; t_2^{i2\pi\theta-1}dt_1 dt_2,$$
\noindent let $u:=\sqrt{t_1t_2}$ and $v:=t_1/t_2$. Then,

$$\quad D_1(\theta,\tau)=\int_{0}^\infty\int_{0}^\infty Q(u \sqrt{\tau} )Q^*( u/\sqrt{\tau}) C_X(v)C^*_X(v) (u \sqrt{v})^{-i2\pi \theta-1} (u/ \sqrt{v})^{i2\pi \theta-1} \frac{u}{v}\; du dv$$

$$=\int_{0}^\infty\int_{0}^\infty Q(u \sqrt{\tau} )Q^*( u/\sqrt{\tau}) C_X(v)C^*_X(v) v^{-i2\pi\theta-1} dv \frac{du}{u}\qquad\quad\quad$$

$$=\int_{0}^\infty |C_X(v)|^2 v^{-i2\pi\theta-1} dv  \int_{0}^\infty Q( u\sqrt{\tau})Q^*(u/\sqrt{\tau}) \frac{du}{u}\;\;\;\qquad\quad\quad\;\;\;$$

\begin{equation}
=(\mathcal{M}|C_X|^2)(i2\pi\theta)\int_{0}^\infty Q( u\sqrt{\tau})Q^*(u/\sqrt{\tau}) \frac{du}{u}.\quad\quad\;\;\;\quad\quad\;\;\;\quad\quad\;\;\;\;\label{d1}
\end{equation}

\noindent For real-valued Gaussian LSSPs, by (\ref{D2}),

$$D_2(\theta, \tau)=\int_{0}^\infty \int_{0}^\infty Q^2(\sqrt{t_1t_2}) C_X(\frac{t_1}{t_2}\tau) C_X(\frac{t_1}{t_2\tau})t_1^{-i2\pi\theta-1} \;t_2^{i2\pi\theta-1}dt_1 dt_2,$$

$$=\int_0^\infty \int_0^\infty Q^2(u) C_X(v\tau)C_X(v/\tau) v^{-i2\pi\theta-1} \frac{du}{u} dv\quad\quad\;$$
$$=\int_0^\infty Q^2(u)\frac{du}{u} \int_0^\infty C_X(v\tau)C_X(v/\tau) v^{-i2\pi\theta-1} dv,\quad\quad$$

\noindent and (\ref{d2}) is obtained.

$\hspace{6in} \boxempty$\\

\noindent\textbf{A5:} Proof of Theorem $9$:
By insertion of (\ref{24}) into (\ref{aex}) and (\ref{19}), we have that
$$A_{E,X}(\theta, \tau)=\int_0^\infty Q(t)C_X(\tau) \tau^{ia(\ln t-b)}t^{-i2\pi\theta-1}dt\qquad$$
$$\qquad\qquad\;\;\;=\tau^{-iab}C_X(\tau)\int_0^\infty Q(t) t^{-i2\pi(\theta-\frac{a}{2\pi}\ln \tau)-1}dt$$
$$\qquad\;\;\;\;\;=\tau^{-iab}C_X(\tau) (\mathcal{M}Q)(i2\pi(\theta-\frac{a}{2\pi}\ln \tau)),$$
$$D_1(\theta, \tau)=\int_0^\infty \int_0^\infty Q(\sqrt{t_1 t_2 \tau}) C_X(\frac{t_1}{t_2}) (\frac{t_1}{t_2})^{ia(\ln \sqrt{t_1 t_2 \tau}-b)}  Q^*(\sqrt{\frac{t_1 t_2}{ \tau}}\;)C^*_X(\frac{t_1}{t_2})(\frac{t_1}{t_2})^{-ia(\ln \sqrt{\frac{t_1 t_2 }{\tau}}-b)}$$
$$\times\; t_1^{-i2\pi\theta-1}\;t_2^{i2\pi\theta-1}dt_1 dt_2$$
$$\qquad\;\;=\int_0^\infty \int_0^\infty Q(u\sqrt{\tau}) Q^*(u/\sqrt{\tau})C_X(v)C^*_X(v)v^{ia(\ln u\sqrt{\tau}-b)}v^{-ia(\ln  \frac {u}{\sqrt{\tau}}-b)} v^{-i2\pi\theta-1} \frac{du}{u} dv,$$
where $u:=\sqrt{t_1 t_2}$ and $v:=\frac{t_1}{t_2}$. So,
$$D_1(\theta, \tau)=\int_0^\infty |C_X(v)|^2 v^{-i2\pi(\theta-\frac{a}{2\pi}\ln \tau)-1} dv \times \int_0^\infty Q(u\sqrt{\tau}) Q^*(u/\sqrt{\tau}) \frac{du}{u}$$
$$\;\;\;=(\mathcal{M}|C_X|^2)(i2\pi(\theta-\frac{a}{2\pi}\ln \tau))\int_0^\infty Q(u\sqrt{\tau}) Q^*(u/\sqrt{\tau}) \frac{du}{u}$$
$$\;\;\;\;=(\mathcal{M}|C_X|^2)(i2\pi(\theta-\frac{a}{2\pi}\ln \tau))\; (Q \varoast Q^*)(\tau).\qquad \qquad\qquad$$
Thus, using (\ref{15}), (\ref{18}) and (\ref{22}), the optimal ambiguity domain kernel for the circularly symmetric Gaussian LSSCP, denoted $\phi_{c.opt}$, is given by:
\begin{equation}
\phi_{c.opt}(\theta, \tau)=\frac{|C_X(\tau)|^2 |(\mathcal{M}Q)(i2\pi(\theta-\frac{a}{2\pi}\ln \tau))|^2}{|C_X(\tau)|^2 |(\mathcal{M}Q)(i2\pi(\theta-\frac{a}{2\pi}\ln \tau))|^2+(\mathcal{M}|C_X|^2)(i2\pi(\theta-\frac{a}{2\pi}\ln \tau))(Q \varoast Q^*)(\tau)}.\label{lsscp}
\end{equation}
\par By (\ref{22}), it is clear that $\phi_{c.opt}(\theta, \tau)=\phi_{opt}(\theta-\frac{a}{2\pi}\ln \tau, \tau)$ where $\phi_{opt}$ is the optimal kernel of the LSSP corresponding to $a=b=0$. Now it follows from $\Phi=\mathcal{M}_1^{-1}\mathcal{M}_2 \phi$ that
$$\Phi_{c.opt}(t,\xi)=\int_0^\infty \int_{-\infty}^\infty \phi_{c.opt}(\theta,\tau) t^{i2\pi\theta} \tau^{-i2\pi\xi-1}d\theta d\tau $$
$$\qquad \qquad\qquad \qquad\;\;=\int_0^\infty \int_{-\infty}^\infty \phi_{opt}(\theta-\frac{a}{2\pi}\ln \tau, \tau) t^{i2\pi\theta} \tau^{-i2\pi\xi-1}d\theta d\tau, $$
\noindent let $ \theta' :=\theta-\frac{a}{2\pi}\ln \tau$, then
$$\Phi_{c.opt}(t,\xi)=\int_0^\infty \int_{-\infty}^\infty \phi_{opt}(\theta ', \tau) t^{i2\pi\theta'} \tau^{-i2\pi(\xi-\frac{a}{2\pi}\ln t)-1}d\theta d\tau $$
$$=\Phi_{opt}(t,\xi-\frac{a}{2\pi}\ln t).\qquad\qquad\quad\;\;\;\;\;\;$$

$\hspace{6in} \boxempty$\\

\noindent\textbf{A6:} \noindent Validity of (\ref{phi1})-(\ref{phi4}): By insertion of functions $Q$ and $C_X$ defined in (\ref{example}) in Eq (\ref{22}), the optimal kernel (\ref{phi1}) is achieved.
$$(\mathcal{M}Q)(i2\pi\theta)=\int_0^\infty Q(t) t^{-i2\pi\theta-1} dt =\int_0^\infty t^{2H-\frac{1}{2}\ln t-i2\pi\theta-1} dt$$
$$\quad\quad\;\;\;\quad\quad\;\;\;\;=\int_0^\infty e^{-\frac{1}{2}u^2+u(2H-i2\pi\theta)}du=\sqrt{2\pi}\;e^{(2H-i2\pi\theta)^2/2},$$
where $u:=\ln t$. So, the numerator of (\ref{22}) is obtained as:
\begin{equation}
\big{|}A_{E,X}(\theta,\tau)\big{|}^2=|C_X(\tau)|^2|(\mathcal{M}Q)(i2\pi\theta)|^2=2\pi\;\tau^{-(c/4)\ln \tau}e^{(2H-i2\pi\theta)^2}. \label{a}
\end{equation}

\noindent Also,
$$(\mathcal{M}|C_X|^2)(i2\pi\theta)=\int_0^\infty\tau^{-\frac{c}{4}\ln\tau}\tau^{-i2\pi\theta-1}d\tau\quad\quad\;\;\;\quad\quad\;\;\;\quad\quad\quad\quad$$
$$\quad\quad\quad\quad\;\;\;\quad\quad\;\;\;=\int_0^\infty\tau^{-\frac{c}{4}\ln\tau-i2\pi\theta-1} d\tau=\int_0^\infty e^{-\frac{c}{4}(u^2+\frac{4}{c}\;i2\pi\theta u)}du$$

$$\quad\quad\quad\quad\;\;\;\;\quad\quad\;\;\;=e^{-\frac{4}{c}(\pi \theta)^2}\int_0^\infty e^{-\frac{c}{4}(u+\frac{4}{c}i\pi\theta)^2 }du=\sqrt{\pi}\; \frac{2}{\sqrt{c}}\;e^{-\frac{4}{c}(\pi \theta)^2},$$
and
$$\int_0^\infty Q(u\sqrt{\tau})Q^*(u/\sqrt{\tau})\frac{du}{u}=\int_0^\infty (u\sqrt{\tau})^{2H-\frac{1}{2} \ln(u\sqrt{\tau})}(\frac{u}{\sqrt{\tau}})^{2H-\frac{1}{2} \ln(u/\sqrt{\tau})}\frac{du}{u}$$
\begin{equation}
\qquad\qquad\qquad\qquad\qquad\qquad=\tau^{-\frac{1}{4}\ln\tau}\int_0^\infty u^{4H-\ln u-1} du=\sqrt{\pi}\; e^{4H^2}\tau^{-\frac{1}{4}\ln\tau}.\label{b1}
\end{equation}
Thus,
\begin{equation}
D_1(\theta, \tau)=(\mathcal{M}|C_X|^2)(i2\pi\theta)\int_{0}^\infty Q( u\sqrt{\tau})Q^*(u/\sqrt{\tau}) \frac{du}{u}= \frac{2\pi}{\sqrt{c}}\;e^{-\frac{4}{c}(\pi \theta)^2+4H^2}\tau^{-\frac{1}{4}\ln\tau}.\label{b}
\end{equation}
By insertion of (\ref{a}) and (\ref{b}) in (\ref{22}), we have that
$$\phi_{opt}(\theta,\tau)=\frac{2\pi\; \tau^{-(c/4)\ln \tau}\;e^{(2H-i2\pi\theta)^2} }{2\pi\; \tau^{-(c/4)\ln \tau}\;e^{(2H-i2\pi\theta)^2} +  \frac{2\pi}{\sqrt{c}}\;e^{-\frac{4}{c}(\pi \theta)^2+4H^2}\tau^{-\frac{1}{4}\ln\tau}}$$

$$=\frac{1}{1+c^{-1/2}\; e^{(1-1/c)(2\pi\theta)^2+8\pi i\theta H}\; \tau^{(\frac{c-1}{4})\ln\tau}}\;.\quad$$


\noindent Validity of Eq (\ref{phi2}):
$$(\mathcal{M}Q)(i2\pi(\theta-\frac{a}{2\pi}\ln \tau))=\int_0^\infty Q(t) t^{-i2\pi(\theta-\frac{a}{2\pi}\ln \tau)-1} dt =\int_0^\infty t^{2H-\frac{1}{2}\ln t-i2\pi(\theta-\frac{a}{2\pi}\ln\tau)-1} dt\quad\quad\quad\quad$$
\begin{equation}
\quad\quad\quad\quad\quad\;\;=\int_0^\infty e^{-\frac{1}{2}u^2+u(2H-i2\pi(\theta-\frac{a}{2\pi}\ln\tau))}du=\sqrt{2\pi}\;e^{(2H-i2\pi(\theta-\frac{a}{2\pi}\ln\tau))^2/2}.\label{a3}
\end{equation}

Then, using (\ref{a3}) and $|C_X(\tau)|^2=\tau^{-(c/4)\ln \tau}$, the numerator of (\ref{lsscp}) becomes
\begin{equation}
\big{|}A_{E,X}(\theta,\tau)\big{|}^2=|C_X(\tau)|^2 |(\mathcal{M}Q)(i2\pi(\theta-\frac{a}{2\pi}\ln \tau))|^2=2\pi\;\tau^{-(c/4)\ln \tau}e^{(2H-i2\pi(\theta-\frac{a}{2\pi}\ln\tau))^2}. \label{a1}
\end{equation}
For the denominator, we have that
$$(\mathcal{M}|C_X|^2)(i2\pi(\theta-\frac{a}{2\pi}\ln\tau))=\int_0^\infty t^{-\frac{c}{4}\ln t} t^{-i2\pi(\theta-a\ln\tau)-1}dt\quad\quad\quad$$
$$\quad\quad\quad\quad\quad\quad\quad\quad\quad=\int_0^\infty e^{-\frac{c}{4}u^2-i2\pi(\theta-\frac{a}{2\pi}\ln\tau)u}du$$
$$\quad\quad\quad\quad\quad\quad\quad\quad\quad\quad=\int_0^\infty e^{-\frac{c}{4}(u^2+\frac{4}{c}i2\pi(\theta-\frac{a}{2\pi}\ln\tau)u)}du$$
\begin{equation}
\quad\quad\quad\quad\quad\quad\quad=2\sqrt{\frac{\pi}{c}}\;e^{-\frac{4\pi^2}{c}(\theta-\frac{a}{2\pi}\ln\tau)^2}.\label{b3}
\end{equation}

\noindent So, by (\ref{b1}) and (\ref{b3}),
\begin{equation}
D_1(\theta,\tau)=(\mathcal{M}|C_X|^2)(i2\pi(\theta-\frac{a}{2\pi}\ln \tau))(Q \varoast Q^*)(\tau)=2\pi c^{-1/2}e^{4H^2-\frac{4\pi^2}{c}(\theta-\frac{a}{2\pi}\ln\tau)^2} \tau^{-\frac{1}{4}\ln \tau}.\label{b2}
\end{equation}
By insertion of (\ref{a1}) and (\ref{b2}) into (\ref{lsscp}), the optimal ambiguity domain kernel (\ref{phi2}) is achieved.

\noindent Validity of Eq (\ref{phi3}):
Inserting the following relations into (\ref{mlsscpkernel}), leads to (\ref{phi3}).
$$(\mathcal{M}Q_j)(i2\pi\theta)=\int_0^\infty t^{2H_j-\frac{1}{2}\ln \tau -i2\pi\theta-1}dt=\sqrt{2\pi} e^{\frac{1}{2}(2H_j-i2\pi\theta)^2},$$

$$\;\quad\quad\mathcal{M}(C_{X_j}C^*_{X_k})(i2\pi\theta)=\int_0^\infty \tau^{-\frac{c_j+c_k}{8}\ln \tau-i2\pi\theta-1} d\tau=\sqrt{2\pi} \frac{2}{\sqrt{c_j+c_k}}\; e^{-\frac{2}{c_j+c_k}(2\pi\theta)^2},$$

$$(Q_j\varoast Q_k^*)(\tau)=\int_0^\infty (u\sqrt{\tau})^{2H_j-\frac{1}{2}\ln u\sqrt{\tau}} (\frac{u}{\sqrt{\tau}})^{2H_k-\frac{1}{2}\ln u/\sqrt{\tau}}\frac{du}{u}.$$


\bibliographystyle{unsrt}

\end{document}